\documentclass[12pt]{amsart}

\usepackage{amssymb}
\usepackage{latexsym}
\usepackage{exscale}
\usepackage{lscape}
\usepackage{mathrsfs}

\usepackage[colorlinks=true, pdfstartview=FitV, linkcolor=blue, citecolor=blue, urlcolor=blue]{hyperref}

\newtheorem{theorem}{Theorem}[section]
\newtheorem{lemma}[theorem]{\bf Lemma}
\newtheorem{corollary}[theorem]{Corollary}
\newtheorem{prop}[theorem]{Proposition}
\newtheorem{definition}[theorem]{Definition}

\theoremstyle{remark}
\newtheorem{remark}[theorem]{\textbf{Remark}}
\newtheorem{example}[theorem]{Example}

\allowdisplaybreaks

\numberwithin{equation}{section}

\newcommand{\R}{\mathbb R}
\newcommand{\Rn}{{{\mathbb R}^n}}

\newcommand{\N}{\mathbb N}

\newcommand{\Z}{\mathbb Z}

\newcommand{\A}{\mathcal A}
\newcommand{\al}{\alpha}
\newcommand{\bxi}{\boldsymbol{\xi}}
\newcommand{\ep}{\varepsilon}

\newcommand{\bU}{\mathbf U}
\newcommand{\bh}{\mathbf h}
\newcommand{\bg}{\mathbf g}
\newcommand{\bw}{\mathbf w}
\newcommand{\bv}{\mathbf v}
\newcommand{\bk}{\mathbf k}
\newcommand{\ba}{\mathbf a}

\newcommand{\vf}{\mathbf f}

\def\H{{\mathscr H}}
\def\W{{\mathscr W}}

\DeclareMathOperator{\supp}{supp}
\DeclareMathOperator{\dist}{dist}

\DeclareMathOperator*{\essinf}{ess\,inf}
\DeclareMathOperator*{\esssup}{ess\,sup}

\newcommand{\tr}{\mathrm{tr}}
\newcommand{\T}{\mathrm{t}}
\newcommand{\op}{\mathrm{op}}
\DeclareMathOperator{\la}{\langle}
\DeclareMathOperator{\ra}{\rangle}

\newcommand{\grad}{\nabla}

\DeclareMathOperator{\Div}{div}

\DeclareMathOperator{\osc}{osc}

\def\Q{{\mathcal Q}}

\def\M{{\mathcal M}}

\def\Lap{{\mathscr L}}
\def\Sp{{\mathcal S}}

\def\Xint#1{\mathchoice
   {\XXint\displaystyle\textstyle{#1}}%
   {\XXint\textstyle\scriptstyle{#1}}%
   {\XXint\scriptstyle\scriptscriptstyle{#1}}%
   {\XXint\scriptscriptstyle\scriptscriptstyle{#1}}%
   \!\int}
\def\XXint#1#2#3{{\setbox0=\hbox{$#1{#2#3}{\int}$}
     \vcenter{\hbox{$#2#3$}}\kern-.5\wd0}}

\def\avgint{\Xint-}
\def\dashint{\Xint-}

\def\ep{{\epsilon}}
\def\loc{\text{loc}}

\def\al{{\alpha}}

\def\W{\mathscr W}

\def\vp{\varphi}

\begin{document}

\title[Matrix $\A_p$ weights for degenerate Sobolev spaces]
{Matrix $\A_p$ weights, Degenerate Sobolev spaces, and mappings of finite distortion}

\author{David Cruz-Uribe, SFO}
\address{David Cruz-Uribe, SFO\\
Dept. of Mathematics \\ Trinity College \\
Hartford, CT 06106-3100, USA} 
\email{david.cruzuribe@trincoll.edu}

\author{Kabe Moen}
\address{Kabe Moen \\ 
Department of Mathematics \\
 University of Alabama \\
 Tuscaloosa, AL 35487, USA}
\email{kabe.moen@ua.edu}

\author{Scott Rodney}
\address{Scott Rodney \\
Department of Mathematics, Physics and Geology \\
 Cape Breton University\\
Sydney, NS B1P6L2, Canada}
\email{scott\_rodney@cbu.ca}

\thanks{The first author is supported by the Stewart-Dowart faculty
  development fund at Trinity College and NSF Grant 1362425.  The second author is supported
  by NSF Grant 1201504.  The third author is supported by the NSERC
  Discovery Grant program.}



\date{May 2, 2015}

\begin{abstract}
We study  degenerate Sobolev spaces where the degeneracy is controlled
by a matrix $\A_p$ weight.  This class of  weights was introduced by Nazarov, Treil
and Volberg, and degenerate Sobolev spaces with matrix weights have been considered by
several authors for their applications to PDEs.  We prove that the
classical Meyers-Serrin theorem, $H=W$, holds in this setting.  As
applications we prove partial regularity results for weak solutions of
degenerate $p$-Laplacian equations, and in particular for mappings of
finite distortion.
\end{abstract}

\maketitle

\section{Introduction}
\label{section:introduction}

In this paper we study matrix $\A_p$ weights and their application to
PDEs and mappings of finite distortion.  Scalar Muckenhoupt $A_p$
weights have a long history: they were introduced in the 1970s and are
central to the study of weighted norm inequalities in harmonic
analysis.  They have extensive applications in PDEs and other areas.
(For details and further references,
see~\cite{duoandikoetxea01,garcia-cuerva-rubiodefrancia85,MR2463316}.)
Matrix $\A_p$ weights are more recent.  They were introduced by
Nazarov, Treil and Volberg~\cite{MR1428988,MR1428818,MR1423034} and
arose from problems in stationary processes and operator theory.  A
matrix weight $W(x)$ is a $d\times d$ semi-definite matrix of
measurable functions.  It is used to define a weighted $L^p$ norm on vector-valued functions:
\[ \|\vf\|_{L^p_W} = \left(\int_{\R^n}
  |W^{1/p}(x)\vf(x)|^p\,dx\right)^{1/p}.  \]
The matrix $\A_p$ condition is a natural generalization of the scalar
Muckenhoupt $A_p$ condition and matrix $\A_p$ weights also share many other analogous
properties of their scalar counterparts.  For instance, the Hilbert transform is bounded on
$L^p_W(\R)$ if and only if $W\in \A_p$.  Since their introduction
these weights have been considered by a number of authors: see, for
instance, \cite{2014arXiv1402.3886B,MR1857041,MR1813604,MR2104276,
  MR2015733, 2014arXiv1401.6570I,MR2354705,MR2670156, MR1928089}.

In this paper we apply the theory of matrix $\A_p$ weights to the
study of degenerate Sobolev spaces. More precisely, we consider the space $\W^{1,p}_W$
that consists of all functions in $W^{1,1}_\loc$ such that
\[ \|f\|_{\W^{1,p}_W} = \|f\|_{L^p(v)} + \|\grad f\|_{L^p_W}<\infty. \]
(The weight $v$ could in principle be arbitrary, but we will show that
there exist scalar weights naturally associated with each matrix weight.)
Such weighted Sobolev spaces are well known to play an important role
in the study of degenerate elliptic equations:  see
\cite{ChuaRodneyWheeden, MR1207810,MR2574880,turesson00}.   Our main result 
extends the celebrated $H=W$ theorem of Meyers and
Serrin~\cite{MR0164252} to Sobolev spaces $\W^{1,p}_W(\Omega)$:  we
will show that if $W\in \A_p$, then smooth functions are dense in $\W^{1,p}_W(\Omega)$.  

We give two applications of our results.  First, we use them to
prove partial regularity results for the degenerate
$p$-Laplacian, 
\[ \Lap_{A,p} u = \Div(\langle A\grad u,\grad
u\rangle^{\frac{p-2}{2}}A\grad u) = 0, \]
where $A$ is an $n\times n$ degenerate elliptic matrix.    These
results extend the work of the first two authors and
Naibo~\cite{MR3011287}; in particular, assuming the matrix $\A_p$
condition allows us to significantly weaken other hypotheses.    Second,
we apply these results for the degenerate
$p$-Laplacian to the
problem of partial regularity of mappings of finite distortion.  Conditions guaranteeing the continuity of such mappings
have been studied by many authors:  see~\cite{MR1241287, MR1833892,
  MR2053566, MR1294334, MR0414869}.  Our results approach the
regularity problem from a significantly different direction.  More precisely, we
characterize the set of continuity of the mapping in terms of a
maximal operator defined using its related inner and outer distortion
functions.

The remainder of this paper is organized as follows.  In
Section~\ref{section:weights} we gather some preliminary material
about scalar weights, particularly the Muckenhoupt $A_p$ weights.
There is a close relationship between scalar $A_p$ and matrix $\A_p$
and the scalar weights play a significant role in our work.  In
Section~\ref{section:matrix} we define matrix weighted spaces and
give some basic results.  None of these ideas are new, but we have
put them a consistent framework and we give proofs for several
results that are only implicit in the literature.  

In Section~\ref{section:matrix-Ap} we define matrix $\A_p$ weights and
prove a number of new results, particularly for matrix $A_1$.  The
central theorem  is that approximate identities converge in $L^p_W$,
$1\leq p < \infty$.  We prove this without using the Hardy-Littlewood
maximal operator, replacing it with a smaller averaging operator.
This fact plays an important role in the proof of our main result, but
it is of independent interest and should be useful in other settings.

In
Section~\ref{section:H=W} we prove our main result, the
generalization of the Meyers-Serrin $H=W$ theorem to matrix weighted Sobolev
spaces.   We prove several variations that correspond to well-known
results in the scalar (unweighted) case.  

The last three sections are applications.  In
Sections~\ref{section:degenerate} and~\ref{section:balance} we apply
our results to degenerate $p$-Laplacian equations. In Section~\ref{section:degenerate} we reformulate and extend the results
in~\cite{MR3011287} without using the matrix $\A_p$
condition and instead give our hypotheses in terms of scalar weights.
In Section~\ref{section:balance} we show that the matrix $\A_p$ condition
yields a number of corollaries.  Finally, in
Section~\ref{section:finite-distortion} we apply these results to
prove partial regularity results for mappings of finite distortion.
All of our results are based on assuming that the distortion tensor
satisfies a matrix $\A_p$ condition.  

Throughout this paper we will use the following notation.  The symbol
$n$ will always denote the dimension of the Euclidean space $\R^n$.
We will use $d$ to denote the dimension of matrix and vector-valued
functions.  In general $d$ can be any positive value, though in
applications we will take $d=n$.  
We will take the domain of our  functions to
be an open, connected set $\Omega\subset \R^n$.  The set $\Omega$ need
not, {\em a priori}, be bounded.  Given two values $A$ and $B$, we will write $A\lesssim
B$ if there exists a constant $c$ such that $A\leq cB$.  We write
$A\approx B$ if $A\lesssim B$ and $B\lesssim A$.   Constants
$C$, $c$, etc., whether explicit or implicit, can change value at each
appearance.  Sometimes we will indicate the parameters constants
depend on by writing, for instance, $C(n,p)$, etc.   If the dependence
is not indicated, the constant may depend on the dimension and other
parameters that should be clear from context.  

\section{Scalar weights}
\label{section:weights}

 In this section we
gather together, without proof, some  basic 
definitions and results about scalar $A_p$ weights.  Unless otherwise noted, these results can be found
in~\cite{duoandikoetxea01,garcia-cuerva-rubiodefrancia85}. 

Given a domain $\Omega \subset \R^n$, we define a (scalar) weight $w$ to be
non-negative function in $L^1_\loc(\Omega)$.      The measure $w\,dx$
is a Borel measure and we define the weighted $L^p$ space,
$L^p(w,\Omega)$, to be the Banach function space with norm
\[ \|f\|_{L^p(w,\Omega)} = \left(\int_{\Rn} |f(x)|^p w(x)\,dx\right)^{1/p}. \]
Given a set $E$, let
\[ w(E) = \int_E w(x)\,dx,  \qquad \avgint_E w(x)\,dx = \frac{1}{|E|}\int_E w(x)\,dx. \]

A weight $w$ is doubling if given any cube $Q$, $w(2Q)\leq Cw(Q)$,
where $2Q$ is the cube with the same center as $Q$ and
$\ell(2Q)=2\ell(Q)$.

For $1<p<\infty$, we
say that $w\in A_p(\Omega)$ if
\[ [w]_{A_p(\Omega)}=\sup_Q\left(\ \dashint_{Q\cap\Omega} w(x)\,dx\right)\left(\ \dashint_{Q\cap\Omega}
  w^{-p'/p}(x)\,dx\right)^{p/p'}<\infty, \]
where the supremum is taken over all cubes $Q$.  
When $p=1$, we say $w\in A_1$ if for all cubes $Q$,
\[ \avgint_{Q\cap\Omega} w(y)\,dy \leq [w]_{A_1} \essinf_{x\in Q\cap \Omega} w(x). \]

\begin{remark} \label{remark:Ap-domain} Alternatively, we can define
  the doubling and $A_p$ conditions with respect to balls instead of
  cubes.  If $\Omega=\R^n$, these two definitions are clearly
  equivalent; similarly, they are equivalent if $w$ is the restriction
  to $\Omega$ of a doubling or an $A_p$ weight defined on all of
  $\R^n$.  However, depending on the geometry of $\Omega$ and its
  boundary these two definitions may not be equivalent.  (For a
  characterization of the restriction problem for $A_p$ weights,
  see~\cite[Chapter~IV.5]{garcia-cuerva-rubiodefrancia85}.)
  Hereafter, given a domain $\Omega$ we will assume that our weights
  are defined on some unspecified set $\Omega'$ such that
  $\Omega \Subset \Omega'$ and we assume that balls and cubes are
  interchangeable in the definition of doubling or $A_p$ on $\Omega$.  Moreover,
  for simplicity, we will write $A_p$ instead of $A_p(\Omega)$: again,
  the precise domain will be implicit.
\end{remark}

Define the class $A_\infty$ by
\[ A_\infty = \bigcup_{p\geq 1} A_p. \]
If $w\in A_p\subset A_\infty$, then for every cube $Q$ and measurable set
$E\subset Q$,
\begin{equation} \label{eqn:Ainfty-Ap}
 \frac{|E|}{|Q|} \leq [w]_{A_p}^{1/p}
\left(\frac{w(E)}{w(Q)}\right)^{1/p}. 
\end{equation}

\medskip

A weight $w$ satisfies the reverse H\"older
condition for some $s>1$, denoted by $w\in RH_s$, if 
\[ [w]_{RH_s} = \sup_Q \left(\avgint_Q
  w(x)^s\,dx\right)^{1/s}\left(\avgint_Q w(x)\,dx\right)^{-1} < \infty.  \]
We say that $w\in RH_\infty$ if for all cubes $Q$,
\[ \esssup_{x\in Q} w(x) \leq [w]_{RH_\infty} \avgint_Q w(y)\,dy. \]
Given a weight $w$,  $w\in A_p$ for some $p$ if and only if
$w\in RH_s$ for some $s$: i.e.,
\[ \bigcup_{1 \leq p<\infty} A_p = A_\infty = \bigcup_{1<s\leq \infty} RH_s. \] 
The reverse H\"older condition yields an estimate that is analogous
to~\eqref{eqn:Ainfty-Ap}, exchanging the roles of Lebesgue measure and
the measure $w\,dx$:  if $w\in RH_s$, then for every cube $Q$ and $E\subset Q$,
\begin{equation} \label{eqn:Ainfty-RHs}
 \frac{w(E)}{w(Q)} \leq [w]_{RH_s}
\left(\frac{|E|}{|Q|}\right)^{1/s'}. 
\end{equation}

Below we will need a sharp estimate for the reverse
H\"older exponent.  The following result is taken from Hyt\"onen and
P\'erez~\cite{HP}.  If $w\in A_\infty$, it satisfies the Fujii-Wilson
condition
\[ [w]_{A_\infty} = \sup_Q \frac{1}{w(Q)}\int_Q M(w\chi_Q)(x)\,dx <
\infty, \]
where $M$ is the Hardy-Littlewood maximal operator,
\[ Mf(x) = \sup_Q \avgint_Q |f(y)|\,dy \cdot \chi_Q(x). \]
Then we have that $w\in RH_s$ with $[w]_{RH_s}\leq 2$, where
\begin{equation} \label{eqn:sharp-RH}
s = 1 + \frac{1}{2^{n+11}[w]_{A_\infty}}. 
\end{equation}

\section{Matrix weighted spaces}
\label{section:matrix}

In this section we define matrix weights and matrix weighted spaces,
and prove some basic properties.    Recall that the symbol $d$ denotes the dimension
of vector functions and matrices:  in other words, we will consider
vector-valued functions 
functions
$\vf : \Omega \rightarrow \R^d$, with
\[ \vf(x) = \big( f_1(x),\ldots,f_d(x)\big), \]
and matrices $A(x) = (a_{ij}(x))_{i,j=1}^d$.   By $D
\vf$ we mean the $n\times d$ matrix $(\partial_i f_j)$.  

Given a vector $\bv=(v_1,\ldots,v_d)$, recall the vector $\ell^p$
norms, $1\leq p<\infty$,
\[ |\bv|_p = \left(\sum_{i=1}^d |v_i|^p\right)^{1/p}, \]
and let $|\bv|_\infty = \max(|v_1|,\ldots,|v_d|)$.  When $p=2$ we will
often write $|\bv|=|\bv|_2$.  We will frequently use the fact that
given $1\leq p<q\leq \infty$,
\[ |\bv|_q \leq |\bv|_p \leq d^{1/p} |\bv|_\infty \leq d^{1/p}|\bv|_q. \]

Let $\M_d$ denote the collection of all real-valued, $d\times d$
matrices.  The norm of a matrix is the operator
norm:
\[ |A|_{\op} = \sup_{\substack{\bv \in \R^d\\|\bv|=1\;\;}} |A\bv|.  \]
A matrix function is  a map $W:\Omega\rightarrow \M_d$; we say that it is
measurable if each component of $W$ is a
measurable function. 

Let $\Sp_d$ denote the collection of all those $A\in \M_d$ that are
self-adjoint and positive semi-definite. 
 If $A\in \Sp_d$, then it has $d$
non-negative eigenvalues, $\lambda_i$, $1\leq i \leq d$, 
and we have that 
\[ |A|_{\op} = \max_i \lambda_i \leq \tr A \leq d|A|_{\op}. \]
Moreover, there exists an orthogonal matrix $U$
such that $U^{\T}AU$ is diagonal.   We denote a diagonal matrix by
$D(\lambda_1,\ldots,\lambda_d)=D(\lambda_i)$.   If
$W$ is a measurable matrix function  with values in $\Sp_d$, then we can
choose the matrices $U(x)$ to be measurable:  the following result is
from~\cite[Lemma~2.3.5]{MR1350650}

\begin{lemma} \label{lemma:diagonal} Given a matrix function
  $W : \Omega \rightarrow \Sp_d$, there exists a $d\times d$ measurable matrix
  function $U$ defined on $\Omega$ such that $U^{\T}(x)W(x)U(x)$ is
  diagonal.
\end{lemma}

If $A\in \Sp_d$ is diagonalized by an
orthogonal matrix $U$ and has eigenvalues $\lambda_i$, for every $s>0$
define  $A^s = U D(\lambda_i^s)U^{\T}$.  By Lemma~\ref{lemma:diagonal}
we have that given any matrix function $W: \Omega\rightarrow \Sp_d$,
$W^s$ is a measurable matrix function.   For a fixed matrix function
$W$ we will always implicitly assume that all of its powers are defined using
the same orthogonal matrix $U$.  Furthermore, if it is the case that $A$ is positive definite we can also define negative powers of $A$ through the orthogonal matrix $U$.  Indeed, a simple calculation shows that $A^{-1}=UD(\lambda_i^{-1})U^t$ and for $s>0$ we set $A^{-s} = UD(\lambda_i^{-s})U^t$.\\

By a matrix weight we mean a matrix function $W : \Omega \rightarrow
\Sp_d$ such that $|W|_{\op} \in
L^1_\loc(\Omega)$.   Equivalently, we may assume that each eigenvalue
$\lambda_i \in L^1_\loc(\Omega)$, $1\leq i \leq d$.   
We say that $W$
is an invertible matrix weight if $W$ is positive definite a.e.:
equivalently, that $\det W(x) \neq 0$ a.e.~and so $W^{-1}$ exists.
Hereafter, if $W$ is a matrix weight, we define 
$v(x) = |W(x)|_\op$; if it is also invertible, we will always let
$w(x) = |W^{-1}(x)|_\op^{-1}$.

\begin{prop} \label{prop:elliptic}
Given an invertible matrix weight $W$, we have
$0<w(x)\leq v(x)<\infty$ for a.e. $x\in \Omega$.  Furthermore,
$W$ satisfies a
two weight, degenerate ellipticity condition: for all $\bxi \in \R^d$,
\begin{equation} \label{eqn:elliptic}
 w(x)|\bxi|^p\leq |W^{1/p}(x)\bxi|^p\leq v(x)|\bxi|^p. 
\end{equation}
\end{prop}

\begin{proof}
First note that for a.e.~$x\in \Omega$,
$ 1 = |I|_\op \leq |W(x)|_\op |W^{-1}(x)|_\op$.
Since $W$ is a matrix weight, $v\in L^1_\loc(\Omega)$; since it is
invertible, its eigenvalues are positive a.e.  Hence, we must have
that $0<w(x)\leq v(x)<\infty$. 

To prove the ellipticity conditions, we use the definition of matrix norm.
The second inequality follows from it immediately:
\[ |W^{1/p}(x)\bxi|^p \leq |W^{1/p}(x)|_\op^p|\bxi|^p =
v(x)|\bxi|^p. \]
The first follows similarly:
$$|\bxi|^p=|W^{-1/p}(x)W^{1/p}(x)\bxi|^p\leq |W^{-1}(x)|_{\op}\, |W^{1/p}(x)\bxi|^p.$$
\end{proof}

\begin{remark} \label{remark:upper-bd}
Note that if $W$ is any matrix weight, the second inequality,
\[ |W^{1/p}(x)\bxi|^p\leq v(x)|\bxi|^p, \]
 still holds.
\end{remark}

\medskip

Given $p$, $1\leq p<\infty$, and a matrix weight $W : \Omega
\rightarrow \Sp_d$,  define the weighted space $L^p_W(\Omega)$ to be the set of all measurable,
vector valued functions $\vf : \Omega \rightarrow \R^d$ such that 
\[ \|\vf\|_{L^p_W(\Omega)} = \left(\int_\Omega
  |W^{1/p}(x)\vf(x)|^p\,dx\right)^{1/p} < \infty. \]
In this space, we identify two functions $\vf,\bg$ as equivalent if $\|\vf-\bg\|_{L^p_W(\Omega)}=0$.  In the special case when $p=2$, it is often useful to restate this norm in terms of the
inner product on $\R^d$:
\[ \|\vf\|_{L^2_W(\Omega)} = \left(\int_\Omega
  \la W(x)\vf(x), \vf(x)\ra \,dx\right)^{1/2}. \]
The following lemma is proved in~\cite{MR2906551,MR2574880}.

\begin{lemma}  \label{lemma:banach}
 Given  $1\leq p<\infty$ and a matrix weight $W:\Omega\rightarrow \Sp_d$, the
  space $L^p_W(\Omega)$ is a Banach space.
\end{lemma}

For a matrix weight that is non-invertible on a set of positive
measure, the equivalence classes of functions can be quite large.
However, if $W$ is invertible, it is straightforward to identify them.

\begin{lemma} \label{lemma:null}
Given $1\leq p < \infty$, an invertible matrix weight $W$, and
$\vf,\,\bg \in L^p_W(\Omega)$, then 
$\|\vf-\bg\|_{L^p_W(\Omega)}=0$ if and only if $\vf(x)=\bg(x)$ a.e.
\end{lemma}

\begin{proof}
Clearly, if $\vf(x)=\bg(x)$ a.e., then $\|\vf-\bg\|_{L^p_W(\Omega)}=0$.
Since $W$ is an invertible matrix weight, we can apply
Proposition~\ref{prop:elliptic} to prove the converse.   By
the ellipticity condition, 
\[ 0 = \|\vf-\bg\|_{L^p_W(\Omega)} \geq
\|\vf-\bg\|_{L^p(w,\Omega)},\]
and since $w(x)>0$ a.e., it follows that $\vf(x)-\bg(x)=\boldsymbol{0}$
a.e.
\end{proof}

The set of bounded functions of compact support, $L_c^\infty(\Omega)$, and
smooth functions of compact support, $C_c^\infty(\Omega)$, are both dense in
$L_W^p(\Omega)$.   These results seem to be known
(cf.~\cite[Theorem~5.1]{MR2015733}) but we have not found proofs in
the literature.  For completeness we include them here.  

\begin{prop} \label{prop:compact-dense}
Given a matrix weight $W:\Omega\rightarrow \Sp_d$, $L_c^\infty(\Omega)$ is dense in $L^p_W(\Omega)$.
\end{prop}

\begin{proof}
First assume that $W(x)$ is diagonal, that is $W(x)=D(\lambda_i(x))$.  Fix $\vf \in
L^p_W(\Omega)$.  Then by the  non-negativity of each $\lambda_i$ and the
equivalence of norms,
\[ \int_\Omega |W^{1/p}(x)\vf(x)|_2^p\,dx
\approx \int_\Omega |W^{1/p}(x)\vf(x)|_p^p\,dx
= \sum_{i=1}^d \int_\Omega  |f_i(x)|^p\lambda_i(x)\,dx. \]
Therefore, we have that $f_i \in L^p(\lambda_i,\Omega)$.  Since
$\lambda_i \in L_\loc^1(\Omega)$, $\lambda_i\,dx$ is a regular Borel measure,
and so $L_c^\infty(\Omega)$ is dense in $L^p(\lambda_i,\Omega)$.  Hence, given
any $\epsilon>0$, there
exists $g_i\in L_c^\infty(\Omega)$ such that
$\|f_i-g_i\|_{L^p(\lambda_i,\Omega)}<\epsilon$.   Let
$\bg=(g_1,\ldots,g_d)$.  By our choice of the $g_i$'s we 
conclude that
\[ \|\vf - \bg \|_{L^p_W(\Omega)} \lesssim \epsilon. \]

Now fix an arbitrary matrix weight $W$ and by
Lemma~\ref{lemma:diagonal} let  $D=U^\T W U$ be its diagonalization.
Let $\vf \in L^p_W(\Omega)$ and set $\bh = U^\T \vf$.  Then by the
orthogonality of $U$,
\[ |D^{1/p}\bh|=|U^\T W^{1/p}UU^\T \vf | = |W^{1/p}\vf|. \]
Hence, $\bh \in L^p_D(\Omega)$ and by the previous argument, for any
$\epsilon>0$, there exists $\bg \in L_c^\infty(\Omega)$ such that 
$\|\bh-\bg\|_{L^p_D(\Omega)}<\epsilon$.   Using orthogonality again,
we have that
\[ |D^{1/p}(\bh - \bg)| = |U^\T W^{1/p}U( U^\T \vf-\bg)|=
|W^{1/p}(\vf - U\bg)|, \]
and since $|U\bg|\leq |U|_{\op}|\bg|$, $U\bg \in L_c^\infty(\Omega)$.  
This completes the proof.
\end{proof}

As a consequence we have that smooth functions are dense in
$L^p_W(\Omega)$. 

\begin{prop} \label{prop:smoothdense} 
Given a matrix weight $W:\Omega\rightarrow \Sp_d$, $C_c^\infty(\Omega)$
is dense in $L^p_W(\Omega)$.  
\end{prop}

\begin{proof} 
Fix $\vf \in L^p_W(\Omega)$ and let $\epsilon>0$.  By Proposition~\ref{prop:compact-dense},
there exists $\bg \in L_c^\infty(\Omega)$ such that
$\|\vf-\bg\|_{L^p_W(\Omega)}<\epsilon/2$.  
Moreover, if we let $v(x)=|W(x)|_{\op}$ then $v\in L^1_\loc(\Omega)$
and 
$$\int_{\Omega}|\bg(x)|^pv(x)\; dx\leq \|\bg\|^p_\infty v(\supp(\bg))<\infty.$$
Thus $|\bg|\in L^p(v,\Omega)$ and, in particular, each component
function of $\bg$ belongs to $L^p(v,\Omega)$.   Therefore, there exists
$\bh \in C_c^\infty(\Omega)$ such that
$\|\bg-\bh\|_{L^p(v,\Omega)}<\epsilon/2$.  By
Remark~\ref{remark:upper-bd}, 
$|W^{1/p}(\bg-\bh)|^p\leq v|\bg-\bh|^p,$
so we can conclude that $\|\vf-\bh\|_{L^p_W(\Omega)}<\epsilon$.
\end{proof}

\section{Matrix $\A_p$}
\label{section:matrix-Ap}

In this section we define matrix $\A_p$ weights and prove some of
their properties.  When $p>1$ they are often defined in terms of norms on
$\R^d$, but  here we take as
our definition an equivalent condition due to Roudenko~\cite{MR1928089}
that more closely resembles the
definition of  scalar $A_p$ weights.  Moreover, this approach also
leads naturally to the definition of matrix $\A_1$, which is due to Frazier
and Roudenko~\cite{MR2104276}.

\begin{definition}  \label{defn:matrix-Ap}
Given $1<p<\infty$, an invertible  matrix weight $W:\Omega\rightarrow
\Sp_d$ is in matrix $\A_p(\Omega)$, denoted by $W\in \A_p(\Omega)$, if
$W^{-p/p'}$ is also a matrix weight and 
$$[W]_{\A_p(\Omega)}=\sup_Q \dashint_{Q\cap\Omega}\left(\ \dashint_{Q\cap\Omega}
  |W^{1/p}(x)W^{-1/p}(y)|_{\op}^{p'}\,dy\right)^{p/p'}dx<\infty,$$
where the supremum is taken over all cubes in
$\R^n$ and where $p'$ is the dual exponent to $p$.  When $p=1$, we say
that $W\in \A_1(\Omega)$ if $W^{-1}$ is a matrix weight and 
$$[W]_{\A_1}=\sup_Q \esssup_{x\in Q} \avgint_{Q\cap\Omega}
  |W(y)W^{-1}(x)|_{\op}\,dy<\infty.$$
\end{definition}

\begin{remark} \label{remark:matrix-Ap-domain}
As is the case  for scalar weights
(cf.~Remark~\ref{remark:Ap-domain}), if $\Omega=\R^n$, then we get an
equivalent definition if we replace cubes with balls.  We will want to
elide between balls and cubes on more general domains.  Therefore, as
in the scalar case, given
any matrix weight $W$ on a domain $\Omega$, we will implicitly assume
that it satisfies the matrix $\A_p$ condition on some larger domain
$\Omega'$ and we will suppress any reference to the domain, writing
$\A_p$ instead of $\A_p(\Omega)$.   We note in passing that the
problem of characterizing those domains $\Omega$ such that every $W\in
\A_p(\Omega)$ is the restriction of a matrix in $\A_p(\R^n)$ is open.  
\end{remark}

\begin{remark}
When $d=1$ and $W(x)=w(x)$ is a scalar valued weight, the matrix
$\A_p$ condition becomes the $A_p$ condition as defined in
Section~\ref{section:weights}.  
\end{remark}

The matrix $\A_p$ weights satisfy the same duality relationship as
scalar $A_p$ weights.  This is due to
Roudenko~\cite[Corollary~3.3]{MR1928089} when $\Omega=\R^n$, but the
proof given there extends without change to the more general setting.  

\begin{lemma}  \label{lemma:Ap-dual}
Given $1<p<\infty$ and  a matrix weight $W$, $W\in \A_p$ if and only if $W^{-p'/p}\in \A_{p'}$.
\end{lemma}

By definition, if $W\in \A_p$ it is an invertible matrix weight, so we
have associated to it the scalar weights $v$ and $w$, and $W$ satisfies the
degenerate ellipticity condition~\eqref{eqn:elliptic}.  Moreover,  these weights
are scalar $A_p$ weights. 

\begin{lemma} \label{lemma:scalar-Ap}
Given  $1 \leq p<\infty$, if $W\in \A_p$,  then $v(x)=|W(x)|_{op}$ and
$w(x) = |W^{-1}(x)|_{\op}^{-1}$  are
scalar $A_p$ weights.
\end{lemma}

\begin{remark}
The converse of this lemma is not true:  for a counter-example, see
Lauzon and Treil~\cite{MR2354705}. 
\end{remark}

\begin{proof}
First suppose that $p>1$.  The fact that $v\in A_p$ is  due to 
Goldberg~\cite[Corollary~2.3]{MR2015733}. (Again, his proof assumes
$\Omega=\R^n$, but it extends to
the general case without change.)  Further, by
Lemma~\ref{lemma:Ap-dual}, $W^{-p'/p}\in \A_{p'}$, so by
the definition of the operator norm and what we just proved,
\[ w^{-p'/p} = |W^{-1}|_{\op}^{p'/p}= |W^{-p'/p}|_{\op} \in A_{p'}. \]
Therefore, by the duality of scalar $A_p$ weights (which follows at
once from the definition), $w\in A_p$. 

\medskip

For the case $p=1$ we modify an argument from Frazier and
Roudenko~\cite[Lemma~2.1]{MR2104276}.  To  prove that $v \in
A_1$ we first
construct a measurable vector function $\bv$ such
that $|\bv(y)|=1$ and $|W(y)|_\op \lesssim|W(y)\bv(y)|$ a.e.
 If
$W=D(\lambda_i)$ is diagonal, let $\bv$  be the
constant vector $\bh=(d^{-1/2},\ldots,d^{-1/2})^\T$.  Then
\[ |D(y)\bh| \geq d^{-1/2}\max_i
\lambda_i(y) = d^{-1/2}|D(y)|_\op. \]
For a general $W$, 
let $D= UWU^\T$ be the diagonalization of $W$ from
Lemma~\ref{lemma:diagonal} and let  $\bv(y)=U^\T(y)\bh$. Then $\bv$ is
measurable and
\begin{equation*} |W(y)\bv(y)| = |U(y)W(y)U^\T(y) \bh|= |D(y)\bh| 
\geq d^{-1/2}|D(y)|_\op = d^{-1/2}|W(y)|_\op. 
\end{equation*}

Given such a vector function $\bv$, we can now estimate as follows.
Fix a cube $Q$, let $x\in Q$, and set $\bw(y)=W(x)\bv(y)$.  Then
\begin{multline*}
\avgint_Q |W(y)|_\op\,dy 
 \lesssim \avgint_Q |W(y)\bv(y)|\,dy
 = \avgint_Q |W(y)W^{-1}(x)\bw(y)|\,dy \\
 \leq \avgint_Q|W(y)W^{-1}(x)|_\op |\bw(y)|\,dy 
 \leq [W]_{\A_1} |W(x)|_\op. 
\end{multline*}

\medskip

To prove that $w\in A_1$, we can argue similarly.  Fix a cube $Q$ and
$x\in Q$.  Arguing as above, construct a vector $\bw=\bw(x)$ so that
$|\bw|=1$ and $|W^{-1}(x)|_\op \lesssim|W^{-1}(x)\bw|$.  Let $\bv=
W^{-1}(x)\bw$.   Then for any $y\in Q$,
\[ |W^{-1}(x)|_\op \lesssim |\bv| = |W^{-1}(y)W(y)\bv| \leq |W^{-1}(y)|_\op |W(y)\bv|. \]
Hence,
\begin{multline*}
|W^{-1}(x)|_\op \avgint_Q |W^{-1}(y)|_\op^{-1}\,dy 
 \lesssim \avgint_Q |W(y)\bv|\,dy \\
 = \avgint_Q |W(y)W^{-1}(x) \bw|\,dy 
 \leq \avgint_Q |W(y)W^{-1}(x)|_\op\,dy 
 \leq [W]_{\A_1}.
\end{multline*}
This completes the proof.
\end{proof}

\medskip

The matrix $\A_p$ condition characterizes the matrix weights $W$ such
that the averaging operators $\vf \mapsto \avgint_Q \vf(x)\,dx$ are
uniformly bounded on $L^p_W(\Omega)$.
(See~\cite[Proposition~2.1]{MR2015733} for the case $p>1$.)  This is
also true for more
general averaging operators.

\begin{prop}  \label{prop:averaging}
Let $\Q$ be a collection of  pairwise disjoint cubes in $\R^n$.  Given
$1\leq p< \infty$ and a matrix
weight $W \in \A_p$,  the averaging operator
\[ A_\Q\vf (x) = \sum_{Q\in \Q} \avgint_Q \vf(y)\,dy \cdot
\chi_Q(x) \]
satisfies
\[ \|A_\Q \vf\|_{L^p_W(\Omega)} \leq
[W]_{\A_p}^{1/p}\|\vf\|_{L^p_W(\Omega)} . \]
\end{prop}

\begin{proof}
To begin, define $\vf \equiv 0$ on $\mathbb{R}^n\setminus
\overline{\Omega}$.  We first consider the case $p>1$:  since the cubes in $\Q$ are disjoint, by
H\"older's inequality and the definition of matrix $\A_p$,
\begin{align}
& \int_\Omega |W^{1/p}(x)A_\Q\vf(x)|^p\,dx \notag\\
&\qquad \quad  \leq \int_{\Rn} |W^{1/p}(x)A_\Q\vf(x)|^p\,dx \notag \\
&\qquad \quad  = \int_{\Rn} \bigg| \sum_{Q\in \Q} \avgint_Q \chi_Q(x)
  W^{1/p}(x) W^{-1/p}(y) W^{1/p}(y)\vf(y)\,dy\bigg|^p
\,dx  \label{eqn:start-here}\\
&\qquad \quad  \leq \int_{\Rn} \sum_{Q\in \Q} \chi_Q(x)
\left(\avgint_Q | W^{1/p}(x) W^{-1/p}(y)|_{\op}^{p'}\,dy\right)^{p/p'}
\left(\avgint_Q | W^{1/p}(y) \vf(y)|\,dy\right)\,dx \notag \\
&\qquad \quad  = \sum_{Q\in \Q} \avgint_Q
\left(\avgint_Q | W^{1/p}(x)
  W^{-1/p}(y)|_{\op}^{p'}\,dy\right)^{p/p'}\,dx 
\left(\int_Q | W^{1/p}(y) \vf(y)|^p\,dy\right) \notag \\
& \qquad \quad \leq [W]_{\A_p} \int_{\Omega} | W^{1/p}(y)
  \vf(y)|^p\,dy. \notag
\end{align}
When $p=1$ the proof is almost identical, omitting H\"older's
inequality and using Fubini's theorem and the definition of $\A_1$.
\end{proof}

We now want to prove that for ``nice'' functions
$\phi\in C_c^\infty(B(0,1))$, the convolution operator
$\vf \mapsto \phi*\vf$ is bounded on $L^p_W(\Omega)$ and that
approximate identities defined using $\phi$ converge.  We first begin
with a lemma.

\begin{lemma} \label{lemma:convo-averages}
Given $1 \leq p<\infty$ and $W\in \A_p$, then for any cube $Q$ and $\vf\in L^p_W(\Omega)$
\[ \| |Q|^{-1}\chi_Q * \vf \|_{L^p_W(\Omega)} \leq C(n,p)[W]_{\A_p}^{1/p}
\|\vf\|_{L^p_W(\Omega)}. \]
The same inequality is true if we replace the cube $Q$ with any ball
$B$.
\end{lemma}

\begin{proof}
Define the cubes 
\[ \{  Q_\bk = Q + \ell(Q)\bk : \bk \in \Z^n \}. \]
The cubes in $Q_\bk$ form a partition of $\R^n$.  Further, we can then divide the cubes $\{3Q_\bk\}$ into $3^n$ families $\Q_j$ of
pairwise disjoint cubes.  But then for every $\bk\in \Z^n$ and $x\in
Q_\bk$, extending $\vf$ by zero in $\mathbb{R}^n\setminus
\overline{\Omega}$,  we have that 
\begin{multline*}
 |W^{1/p}(x) |Q|^{-1}\chi_Q * \vf (x)|
= \bigg|  |Q|^{-1} \int_{\R^n} W^{1/p}(x) \vf(y)\chi_Q(x-y)\,dy \bigg| \\
\leq |Q|^{-1} \int_{\R^n} |W^{1/p}(x) \vf(y)\chi_Q(x-y)|\,dy 
\leq C(n) \avgint_{3Q_\bk}|W^{1/p}(x) \vf(y)|\,dy. 
\end{multline*}
Therefore, 
\begin{multline*}
 \int_{\Rn}|W^{1/p}(x) |Q|^{-1}\chi_Q * \vf (x)|^p\,dx 
\leq C(n,p) \sum_{\bk \in \Z^n} \int_{Q_\bk}
\bigg(\avgint_{3Q_\bk} |W^{1/p}(x) \vf(y)|\,dy\bigg)^p\,dx \\
 \leq C(n,p)\sum_{j=1}^{3^n}\sum_{Q\in \Q_j} 
\int_{Q}\bigg(\avgint_{Q} |W^{1/p}(x) \vf(y)|\,dy\bigg)^p\,dx,
\end{multline*}
and we can now argue exactly as in the proof of
Proposition~\ref{prop:averaging}, starting at~\eqref{eqn:start-here}, to get the desired estimate for
cubes.

To prove this for balls, fix a ball $B$, and let $Q$ be the smallest
cube containing $B$.  Then $|B|\approx |Q|$, and arguing as above, we
get
\[ |W^{1/p}(x) |B|^{-1}\chi_B * \vf (x)|\leq C(n)
\avgint_{3Q_\bk}|W^{1/p}(x) \vf(y)|\,dy, \]
and the proof continues as before.
\end{proof}

\begin{theorem} \label{thm:smooth-convo}
Given $1\leq p < \infty$ and $W\in \A_p$, let $\phi\in C_c^\infty(B(0,1))$ be a non-negative, radially symmetric
and decreasing
function with $\|\phi\|_{L^1(\mathbb{R}^n)}=1$, and for $t>0$ let
$\phi_t(x)=t^{-n}\phi(x/t)$.  Then 
\begin{equation} \label{eqn:smooth-convo1} 
 \sup_{t>0} \|\phi_t * \vf\|_{L^p_W(\Omega)} \leq
C(n,p)[W]_{\A_p}^{1/p}\|f\|_{L^p_W(\Omega)} 
\end{equation}
for every $\vf\in L^p_W(\Omega)$.  As a consequence, we have that
for every such $\vf$,
\begin{equation} \label{eqn:smooth-convo2}
 \lim_{t\rightarrow 0} \|\phi_t * \vf-\vf\|_{L^p_W(\Omega)} = 0. 
\end{equation}
\end{theorem}

\begin{proof}
To prove~\eqref{eqn:smooth-convo1}, consider the function 
\[ \Phi(x) = \sum_{k=1}^\infty a_k |B_k|^{-1}\chi_{B_k}(x), \]
where the balls $B_k$ are centered at the origin, $B_{k+1}\subset B_k$
for all $k$, and the $a_k$ are non-negative with $\sum a_k =1$.  Extending $\vf$ by zero as before it
will suffice to show that 
\[ \|\Phi*\vf\|_{L^p_W(\mathbb{R}^n)} \leq
C(n,p)[W]_{\A_p}^{1/p}\|\vf\|_{L^p_W(\mathbb{R}^n)}; \]
inequality~\eqref{eqn:smooth-convo1} follows by approximating $\phi_t$
from below by a sequence of such functions and applying Fatou's lemma.
But by Minkowski's inequality and Lemma~\ref{lemma:convo-averages},
\begin{multline*}
 \|\Phi*\vf\|_{L^p_W(\mathbb{R}^n)} \leq 
\sum_{k=1}^\infty a_k \||B_k|^{-1} \chi_{B_k}*\vf\|_{L^p_W(\mathbb{R}^n)} \\
\leq C(n,p)[W]_{\A_p}^{1/p} \sum_{k=1}^\infty a_k \|\vf\|_{L^p_W(\mathbb{R}^n)}
= C(n,p)[W]_{\A_p}^{1/p} \|\vf\|_{L^p_W(\mathbb{R}^n)}. 
\end{multline*}

To prove~\eqref{eqn:smooth-convo2}, fix $\epsilon>0$.  Given $\vf \in
L^p_W(\Omega)$, by Proposition~\ref{prop:smoothdense} there exists
$\bg \in C_c^\infty(\Omega)$ such that
$\|\vf-\bg\|_{L^p_W(\Omega)}<\epsilon$.    By a classical result we
have that $\phi_t*\bg \rightarrow \bg$ uniformly, and so
by~\eqref{eqn:elliptic} for all $t$
sufficiently small, 
\[ \|\phi_t*\bg -\bg\|_{L^p_W(\Omega)} \leq
\|\phi_t*\bg -\bg\|_{L^p(v,\Omega)} < \epsilon. \]
Therefore, by~\eqref{eqn:smooth-convo1} we have that
\begin{multline*}
\|\phi_t*\vf - \vf\|_{L^p_W(\Omega)}
\leq \|\phi_t*\bg - \bg\|_{L^p_W(\Omega)} +
\|\phi_t*\vf - \phi_t*\bg\|_{L^p_W(\Omega)}
+ \|\vf-\bg\|_{L^p_W(\Omega)} \\
< \epsilon + C\|\vf-\bg\|_{L^p_W(\Omega)} \lesssim \epsilon.
\end{multline*}
\end{proof}

\begin{remark}
In our proof of Theorem~\ref{thm:smooth-convo} the restrictions on
$\phi$ seem artificial when compared to the scalar case, where any
non-negative function $\phi \in C_c^\infty$ can be used.  We need our
restrictions to allow us to approximate $\phi$ by step functions like
$\Phi$.  It is also possible to prove 
inequality~\eqref{eqn:smooth-convo1} by appealing to the bounds for singular
integrals given in~\cite{MR2015733}.  This approach only works for $p>1$,
but does allow for a larger class of functions $\phi$.   Details are
left to the interested reader.    This was the approach we used in an
early version of this paper; we want to thank S.~Treil for suggesting
the idea behind the proof we give above.
\end{remark}

\section{Degenerate Sobolev spaces and $H=W$}
\label{section:H=W}

In this section we define a family of degenerate Sobolev spaces
using the matrix weighted spaces $L^p_W(\Omega)$.  As we noted above, such
spaces have been studied previously; 
here we consider them in the particular cases where $W$ is either an invertible
matrix weight or a matrix $\A_p$ weight.   

Hereafter, let $W\in \Sp_n$ be an invertible matrix
weight and let $v(x)=|W(x)|_\op$ and $w(x)=|W^{-1}(x)|_\op^{-1}$.
For $1\leq p< \infty$, define the degenerate Sobolev space
$\W^{1,p}_{W}(\Omega)$ to be the set of all $f\in
\W_{\loc}^{1,1}(\Omega)$ such that
\[  \|f\|_{\W^{1,p}_{W}(\Omega)}
=\|f\|_{L^p(v,\Omega)}+\|\nabla f\|_{L^p_W(\Omega)} < \infty. 
\]
Viewing this space as a collection of pairs of the form $(f,\nabla f)$, it is clear that we may consider $\W^{1,p}_W(\Omega)$ as a linear subspace of the Banach space $ L^p(v,\Omega)\oplus
L^p_W(\Omega)$:  since $v\in L^1_\loc(\Omega)$,  $L^p(v,\Omega)$ is a Banach
space and by  Lemma~\ref{lemma:banach} so is
$L^p_W(\Omega)$.   Clearly,  $\W^{1,p}_W(\Omega)$ is non-trivial:  for instance, if $f\in C_c^\infty(\Omega)$,
then $f\in \W_W^{1,p}(\Omega)$, since by Proposition~\ref{prop:elliptic},
$$\|f\|_{\W^{1,p}_{W}(\Omega)}\leq (\|f\|_\infty+\|\nabla
f\|_\infty)v(\supp(f))<\infty.$$ 

Matrix weighted Sobolev spaces generalize the scalar weighted Sobolev
spaces:  that is, given a weight $u$, the space $\W^{1,p}(u,\Omega)$
of functions in $\W^{1,1}_{\rm loc}(\Omega)$ such that
$$\|f\|_{\W^{1,p}(u,\Omega)}=\|f\|_{L^p(u,\Omega)}+\|\nabla
f\|_{L^p(u,\Omega)}<\infty.$$
Every matrix weighted space $\W_W^{1,p}(\Omega)$ is nested between two
scalar weighted spaces.  
By Proposition \eqref{prop:elliptic}, we have that
\begin{equation*}
\|f\|_{\W^{1,p}(w,\Omega)}\leq \|f\|_{\W^{1,p}_W(\Omega)}\leq
\|f\|_{\W^{1,p}(v,\Omega)}; 
\end{equation*}
hence, 
 $$\W^{1,p}(v,\Omega)\subset \W_W^{1,p}(\Omega)\subset \W^{1,p}(w,\Omega).$$
In general, these inclusions are proper as the following example shows. 

\begin{example} \label{example:inclusion}
Let $\Omega=(0,1)\times(0,1)$.  Fix $1<p<\infty$ and $\alpha\in
(0,1)$.   Define the matrix weight
$$W(x,y)=\left[\begin{array}{cc} 1 & 0 \\ 0 & x^{-\al}y^{-\al}
\end{array}\right].$$
Then a straightforward calculation shows that $W\in \A_p$ since it is
a diagonal matrix whose entries are the product of scalar $A_1$
weights in each independent variable.   It is also easy to see that
the weights $v$ and $w$ are given by
\[  v(x,y) = x^{-\alpha}y^{-\alpha}, \qquad w(x,y) = 1. \]
Clearly, $v\,w\in L^1(\Omega)$.  Now define two elements of
$\W^{1,1}_{\text{loc}}(\Omega)$: $f(x,y) = cx^{\frac{\al-1}{p}+1}$ and
$g(x,y) = cy^{\frac{\al-1}{p}+1}$, where $c>0$ is chosen so that
$$\nabla f(x,y) = 
\left[\begin{array}{c} x^\frac{\al-1}{p} \\ 0\end{array}\right]
\;\text{ and }\; \nabla g(x,y)=\left[\begin{array}{c} 0 \\
                                   y^\frac{\al-1}{p}\end{array}\right].$$
Since $f,g$ are bounded,  $f\in L^p(v,\Omega)$ and
$g\in L^p(w,\Omega)$.  The gradient of $f$ satisfies
\begin{gather*}
\int_0^1\int_0^1 |W^{1/p}\nabla f|^p_p\,dxdy
=\int_0^1\int_0^1 x^{\al-1}\,dxdy=\frac{1}{\al}<\infty,\\
\int_{0}^1\int_0^{1}|\nabla f|_p^pv\,dxdy
=\int_{0}^1y^{-\al}\;dy\int_0^{1}\frac{1}{x}\,dx=\infty.
\end{gather*}
The opposite holds for $g$: that is
$\|\nabla g\|_{L^p_W(\Omega)}^p =\infty$ and
$\|\nabla g\|_{L^p(w,\Omega)}^p = \frac{1}{\al}$.  Thus,  $f$ belongs to
$\W_W^{1,p}(\Omega)\setminus\W^{1,p}(v,\Omega)$,  while $g$ belongs to
$\W^{1,p}(w,\Omega)\setminus\W^{1,p}_W(\Omega)$.
\end{example}

Essential to our results is the requirement that $\W^{1,p}_W(\Omega)$
be a Banach space.  This is achieved by imposing size conditions on
$w^{-1}$ as the next theorem demonstrates.

\begin{theorem} \label{thm:complete}
Given a domain $\Omega$,  $1 \leq p<\infty$ and an invertible matrix
weight $W$, suppose $w^{-p'/p} \in L^1_{\emph{loc}}(\Omega)$ (if $p=1$,
$w^{-1}\in L^\infty_{\emph{loc}}(\Omega)$).  Then $\W_W^{1,p}(\Omega)$ is a
Banach space.   In particular, this is the case if $1\leq p<\infty$ and
$W\in \A_p$.
\end{theorem}

\begin{proof}
We need to  show that $\W_W^{1,p}(\Omega)$ is a closed subspace of $ L^p(v,\Omega)\oplus L^p_W(\Omega)$.  
Fix a Cauchy sequence $\{u_k\}$ in $\W_W^{1,p}(\Omega)$.   
Then 
there exists $u\in L^p(v,\Omega)$ and $ \bU \in L^p_W(\Omega)$ such that
$u_k\rightarrow u$ in $L^p(v,\Omega)$ and $\nabla u_k\rightarrow \bU$ in 
$L^p_W(\Omega)$.  We will show that $u,\,\bU \in L^1_\loc(\Omega)$ and that 
$\bU=\grad u$ in the sense of distributional derivatives.    Then
$u\in W^{1,1}_\loc(\Omega)$ with $u\in L^p(v,\Omega)$ and $\nabla u\in L^p_W(\Omega)$.  Thus $u$ belongs to $\W_W^{1,p}(\Omega)$.
 
Fix $\vp\in C_c^\infty(\Omega)$; we need to show that 
\[
\int_\Omega U_j\vp \,dx=-\int_\Omega u\partial_j\vp\,dx, \qquad
1\leq j \leq n.
\]
These integrals are finite since $u$ and $\bU$ are locally
integrable.  To see this suppose first that $p>1$.  Let $K=\text{supp}(\varphi) \Subset \Omega$;
then, since
$w^{-p'/p} \in L^1_\loc(\mathbb{R}^n)$, we have that
\begin{multline*}
\int_\Omega |U_j\varphi| \,dx\leq \|\varphi\|_\infty \int_K |U_j|w^{1/p}w^{-1/p}\,dx \\
\leq w^{-p'/p}(K)^{1/p'}\|\varphi\|_\infty\left(\ \int_\Omega |U_j|^pw\,dx\right)^{1/p}
\leq w^{-p'/p}(K)^{1/p'} \|\varphi\|_\infty\|\bU\|_{L^p_W(\Omega)}.
\end{multline*}
We can bound the other integral similarly:  since $w\leq v$ a.e.,
$v^{-p'/p}\in L^1_\loc(\Omega)$ and so
$$\int_\Omega |u\partial_j\varphi |\,dx\leq
v^{-p'/p}(K)^{1/p'}\|\nabla \varphi\|_\infty \|u\|_{L^p(v,\Omega)}.$$
When $p=1$, we can argue similarly, using the fact that 
$v^{-1},\,w^{-1}\in L^\infty_\loc(\Omega)$. 

\medskip

We now show that these two integrals are equal.   
With $\varphi$ and $K$ as before, by the weak differentiability of
each $u_k$ we have that
\begin{multline*}
\bigg|\int_\Omega U_j\vp+u\partial_j\vp\,dx\bigg|
=\bigg|\int_\Omega ( U_j-\partial_ju_k)\vp\,dx+\int_\Omega (u-u_k)\partial_j\vp\,dx\bigg|\\
\leq  w^{-p'/p}(K)^{1/p'}\|\vp\|_\infty\|\bU-\nabla u_k\|_{L^p_W(\Omega)} 
 +  v^{-p'/p}(K)^{1/p'}\|\partial_j\vp\|_\infty\|u_k-u\|_{L^p(v,\Omega)}.\\
\end{multline*}
Both terms on the right go to zero as $k\rightarrow \infty$.  Thus we
have shown that $\bU=\nabla u$ in the sense of distributional
derivatives and so $u\in \W^{1,1}_{\loc}(\Omega)$.   

Finally, note that if $p>1$ and $W\in \A_p$, then by
Lemma~\ref{lemma:scalar-Ap}, $w\in A_p$ and so $w^{-p'/p}\in A_{p'}$
and thus is locally integrable.   When $p=1$, it follows from
the fact that $w\in A_1$ that $w$ is locally
bounded away from zero and so $w^{-1}$ is locally bounded.
This completes
the proof.
\end{proof}

The importance of the matrix $\A_p$ condition is that it lets us
prove, as is the case  in the classical Sobolev spaces, that smooth
functions are dense in $\W_W^{1,p}(\Omega)$.  Define
$\H_W^{1,p}(\Omega)$ to be the closure of
$C^\infty(\Omega)\cap \W_W^{1,p}(\Omega)$ in $\W_W^{1,p}(\Omega)$.

\begin{theorem} \label{thm:H=W}
Given a domain $\Omega$, if $1\leq p<\infty$ and $W\in \A_p$, then
$$\W_W^{1,p}(\Omega)=\H_W^{1,p}(\Omega).$$
\end{theorem}

\begin{remark} 
The assumption that $W\in \A_p$ is sharp for the conclusion of
Theorem~\ref{thm:H=W} to hold.  To show this, we sketch ~\cite[Example~3.9]{MR2560047} for the case $p=2$.  There, the authors consider the matrix
$$A=\left(\begin{array}{lc} |x|^{2\gamma} &0\\
\;0& 1
\end{array}\right)
$$
for
$x\in\Omega = [-\frac{1}{3},\frac{1}{3}]\times
[-\frac{1}{3},\frac{1}{3}]$
with $\gamma >0$.  It is clear that $A\in \A_2$ for $0<\gamma<1/2$
while $A\notin\A_2$ for $\gamma>1/2$.  In the latter case, the
function $u(x)=|x|^\alpha$ with
$\alpha \in (\max\{-\frac{1}{2},\frac{1-2\gamma}{2}\},0)$ is shown to
be a member of $\H_A^{1,2}(\Omega)$ while its gradient
$\nabla u = (\alpha |x|^{\alpha-1}x,0)$ is not an
$L^1_{\rm{loc}}(\Omega)$ function and hence
$u\notin \W^{1,2}_A(\Omega)$.
\end{remark}

\begin{proof} 
We will  show that $\W^{1,p}_W(\Omega)\subset \H^{1,p}_W(\Omega)$
since the reverse inclusion holds by definition.  The proof is an
adaption of the classic proof that $H=W$:  see~\cite{MR2424078,MR0164252}.   We will show that given any $f\in
\W^{1,p}_W(\Omega)$ and any $\epsilon>0$, there exists $g\in
C^\infty(\Omega)\cap \W^{1,p}_W(\Omega)$ such that
$\|f-g\|_{\W^{1,p}_W (\Omega)}<\epsilon$.  

For each $j\in \N$, define the bounded sets
\[ \Omega_j = \{ x\in \Omega : |x|<j, \dist(x,\partial\Omega)>1/j
\}. \]
Let $\Omega_0=\Omega_{-1}=\emptyset$ and  define the sets
$A_j = \Omega_{j+1}\setminus \overline{\Omega}_{j-1}$.  These sets are
an open cover of $\Omega$, each $\overline{A_j}$ is compact, and given $x\in \Omega$, $x\in
A_j$ for only a finite number of indices $j$.   We can therefore form
a partition of unity subordinate to this cover:  there exists
$\psi_j\in C_c^\infty(A_j)$ such that for all
$x\in \Omega$, $0\leq \psi_j(x)\leq 1$ and
\[ \sum_{j=1}^\infty \psi_j(x) = 1. \]
Since $f\in \W^{1,1}_\loc(\Omega)$,
$\psi_j f \in \W^{1,1}_\loc(\Omega)$.  Furthermore, since
$\grad(\psi_j f) = \psi_j \grad f + f\grad \psi_j$ a.e. in $\Omega$
(see \cite[Section~7.3]{MR1814364}), we have that
\begin{multline*}
 |W^{1/p} \grad(\psi_j f)| \leq |\psi_j||W^{1/p} \grad f| +
|f||W^{1/p} \grad \psi_j| \\
\leq 
\|\psi_j\|_\infty|W^{1/p} \grad f| + \|\nabla\psi_j\|_\infty
|f|v^{1/p}, 
\end{multline*}
and so $\psi_j f \in \W^{1,p}_W(\Omega)$.  

Fix a non-negative, radially symmetric and decreasing function 
$\phi \in C_c^\infty(B(0,1))$ with $\int \phi\,dx=1$.   Then the
convolution
\[ \phi_t*(\psi_jf)(x) = \int_{A_j} \phi_t(x-y)\psi_j(y)f(y)\,dy \]
is only non-zero if for some $y\in A_j$,  $|x-y|<t$.

Hence, for $j\geq 3$, if we fix $t=t_j$, $0<t_j<(j+1)^{-1}-(j+2)^{-1}$, this will hold only if
$(j+2)^{-1} < \dist(x,\partial\Omega) \leq (j-2)^{-1}$.  Therefore,
\[  \supp(\phi_{t_j}*(\psi_jf)) 
\subset \Omega_{j+2}\setminus \overline{\Omega}_{j-2}
= B_j \Subset \Omega. \]
We will fix the precise value of $t_j$ below.  

Define 
\[ g(x)  = \sum_{j=1}^\infty \phi_{t_j}*(\psi_jf)(x).  \]
Since $\phi$ is smooth, each summand is in $C^\infty(\Omega)$.
Further, given $x\in \Omega$, it is contained in a finite number of the
$B_j$, so only a finite number of terms are non-zero.  Thus the series
converges locally uniformly and $g\in
C^\infty(\Omega)$.    

Finally, fix $\epsilon>0$; we claim that for the appropriate
choice of $t_j$ we have $\|f-g\|_{\W^{1,p}_W(\Omega)}<\epsilon$.
To prove this, we consider each part of the norm separately.   Since
$v\in A_p$, the approximate identity $\{\phi_t\}_{t>0}$ converges in
$L^p(v,\Omega)$.  (See~\cite[Theorem~2.1.4]{turesson00}.)   Therefore, for each $j$ there
exists $t_j$ such that
\[ \|f-g\|_{L^p(v,\Omega)} 
\leq \sum_{j=1}^\infty \|\psi_jf - \phi_{t_j}*(\psi_jf)\|_{L^p(v,\Omega)} 
\leq \sum_{j=1}^\infty \frac{\epsilon}{2^{j+1}} = \frac{\epsilon}{2}. \]

The argument for the second part of the norm is similar.
Since $\psi_jf \in \W^{1,1}_\loc(\Omega)$, $\phi_{t_j} *\grad(\psi_jf)=
\grad(\phi_{t_j} *\psi_j f)$.    Fix $j$; then by
Theorem~\ref{thm:smooth-convo} there exists $t_j$ such that 
\[ \|\grad\big(\psi_jf - \phi_{t_j}*(\psi_jf)\big)\|_{L^p_W(\Omega)} 
 = \|\grad(\psi_jf) - \phi_{t_j}*\grad(\psi_jf)\|_{L^p_W(\Omega)} 
< \frac{\epsilon}{2^{j+1}}.  \]
Therefore,
\[ \|\grad(f-g)\|_{L^p_W(\Omega)}
\leq \sum_{j=1}^\infty \|\grad\big(\psi_jf - \phi_{t_j}*(\psi_jf)\big)\|_{L^p_W(\Omega)}
\leq \sum_{j=1}^\infty \frac{\epsilon}{2^{j+1}} = \frac{\epsilon}{2}. \]
Thus, we have shown that $\|f-g\|_{\W^{1,p}_W(\Omega)}<\epsilon$ and
our proof is complete.
\end{proof}

As a corollary to Theorem~\ref{thm:H=W} we can prove that when
$\Omega=\R^n$, smooth functions of compact support are dense.

\begin{corollary} \label{cor:compact-dense}
If $1\leq p< \infty$ and $W\in \A_p$, then $C_c^\infty(\R^n)$  is
dense in $\W^{1,p}_W(\R^n)$. 
\end{corollary}

\begin{proof}
Fix $\epsilon>0$ and $f\in \W_W^{1,p}(\R^n)$.  By Theorem~\ref{thm:H=W} there
exists $h \in C^\infty(\Omega)\cap \W_W^{1,p}(\R^n)$ such that
\[ \|f-h\|_{\W_W^{1,p}(\R^n)} < \epsilon/2. \]
Therefore, to complete the proof we will construct $g\in
C_c^\infty(\R^n)$ such that 
\begin{equation} \label{eqn:smooth-approx}
\|g-h\|_{\W_W^{1,p}(\R^n)} < \epsilon/2. 
\end{equation}

For each $k\geq 2$, let $\nu_k \in C_c^\infty(\R^n)$ be such that
$\supp(\nu_k) \subset B(0,2k)$, $0\leq \nu_k\leq 1$, $\nu_k(x)=1$ for
$x\in B(0,k)$, and $|\grad \nu_k|\lesssim1/k$.  Let $g_k=h\nu_k$.
Then $g_k\in C_c^\infty(\R^n)$, $|g_k|\leq |h|$, and
$g_k\rightarrow h$ pointwise as $k\rightarrow \infty$.  Since
$h\in L^p(v,\R^n)$, by the dominated convergence theorem,
\[ \lim_{k\rightarrow \infty}\|g_k-h\|_{L^p(v,\R^n)} = 0. \]
Similarly, since $\grad g_k= \nu_k\grad h + h \grad \nu_k$, $\grad g_k
\rightarrow \grad h$ as $k\rightarrow \infty$.  Furthermore,
by~\eqref{eqn:elliptic}, 
\begin{multline*}
 |W^{1/p}(x)\grad g_k(x)|^p
\lesssim |\nu_k W^{1/p}(x)\grad h(x)|^p 
+ |h W^{1/p}(x)\grad \nu_k(x)|^p \\
\lesssim |W^{1/p}(x)\grad h(x)|^p + |\grad\nu_k(x)|^p|h(x)|^pv(x)
\lesssim |W^{1/p}(x)\grad h(x)|^p +|h(x)|^pv(x). 
\end{multline*}
Since $h\in \W^{1,p}_W(\R^n)$, the final term is in $L^1(\R^n)$, so
again by the dominated convergence theorem 
\[ \lim_{k\rightarrow \infty}\|\grad g_k-\grad h\|_{L^p_W(\R^n)} = 0. \]
Therefore, for $k$ sufficiently large, if we let $g=g_k$, we get
inquality~\eqref{eqn:smooth-approx} as desired.
\end{proof}

By modifying the proof of Theorem~\ref{thm:H=W} we can also show that
functions that are smooth up the boundary are dense in
$\W^{1,p}_W(\Omega)$ provided $\Omega$ has some boundary regularity.
Given a bounded domain $\Omega$, let $\Sp_W^{1,p}(\Omega)$ denote the
closure of $C^\infty(\overline{\Omega})$ in $\W^{1,p}_W(\Omega)$.

\begin{theorem} \label{thm:bndry-smooth}
Let $\Omega$ be a bounded domain such that $\partial \Omega$ is locally a
Lipschitz graph.  Then for $1\leq p<\infty$ and $W\in \A_p$,
$\Sp_W^{1,p}(\Omega)=\W^{1,p}_W(\Omega)$.  
\end{theorem}

\begin{proof}
The proof of this result in the classical case (see, for instance,
Evans and Gariepy~\cite[Section~4.2]{MR1158660}) is an adaptation of
the proof that $H=W$.  In our setting, we can use the same
modifications to adapt the proof of Theorem~\ref{thm:H=W} and we leave
the details to the reader.  Here, we note that the heart
of the changes is proving that, given a fixed vector $\ba \in \R^n$,  $\phi_t*f(\cdot
+\ba t)$ converges to $f$ in
$\W^{1,p}_W(\Omega)$.  To modify the argument given above, it will suffice to
prove that 
\[ \sup_{t>0} \|\phi_t*f(\cdot+\ba t)\|_{L^p_W(\Omega)} \leq
C\|f\|_{L^p_W(\Omega)}. \]
But if we fix $t>0$, 
\[ \phi_t*f(x+\ba t) = \int_\Omega t^{-n}
\phi\left(\frac{x-y}{t}+\ba\right) f(y)\,dy = \psi*f(x), \]
where $\psi$ is a positive, radially decreasing function centered at
$\ba$.   Such $\psi$ can be approximated by functions of the form
\[ \Phi(x) = \sum_{k=1}^\infty a_k|B_k|^{-1}\chi_{B_k}(x), \]
where the balls $B_k$ are nested and centered at $\ba$.  With such
functions $\Phi$, the proof of Theorem~\ref{thm:smooth-convo}
goes through without change.  
\end{proof}

\section{Degenerate $p$-Laplacian equations}
\label{section:degenerate}

We now consider the applications of matrix weighted Sobolev spaces 
to the study of degenerate elliptic equations.  In this section we
generalize some results from~\cite{MR3011287} for
arbitrary matrix weights; in Section~\ref{section:balance} we will
apply these results in the special case when we assume the
matrix $\A_p$ condition.   Throughout this section, let $\Omega$ be a
bounded domain in $\R^n$.   

In~\cite{MR3011287} the authors studied the partial regularity of
solutions to the divergence form degenerate $p$-Laplacian
\begin{equation}\label{ourPDE} 
\Lap_{A,p} u=\Div(\langle A\nabla u,\nabla u\rangle^{\frac{p-2}{2}}A\nabla u)=0,
\end{equation}
where $1<p<\infty$ and $A\in \Sp_n$  satisfies the
ellipticity condition
$$w(x)^{2/p}|\boldsymbol{\xi}|^2\leq \langle
A\boldsymbol{\xi},\boldsymbol{\xi}\rangle\leq
v(x)^{2/p}|\boldsymbol{\xi}|^2,$$
where the weights $v,\,w$ are assumed to be locally integrable.  In
the  terminology introduced above, we have that $A$ is a
matrix weight.     We want to recast this equation so that our results
can be restated in terms of the degenerate Sobolev spaces defined
in Section~\ref{section:H=W}.    If we define the matrix weight $W$ by $A^{1/2}=W^{1/p}$, then
\eqref{ourPDE}
becomes
\begin{equation} \label{eqn:new-PDE}
\Lap_{W,p} u=\Div(|W^{1/p}\nabla u|^{p-2}W^{2/p}\nabla u)=0
\end{equation}
with ellipticity condition
\begin{equation} \label{ellipcond} 
w(x)|\boldsymbol{\xi}|^p\leq |W^{1/p}(x)\boldsymbol{\xi}|^p\leq v(x)|\boldsymbol{\xi}|^p.
\end{equation}
Hereafter, we will assume that $W$ is an invertible matrix weight and
we will generally assume that $v=|W|_\op$ and $w=|W^{-1}|_\op^{-1}$.
Since $v$ and $w$ are the largest and smallest eigenvalues of $W$,
this choice is in some sense optimal.

As in~\cite{MR3011287}, we  introduce the notion of
$p$-admissible pairs of weights on $\Omega$.  Following the convention
introduced above in Remarks~\ref{remark:Ap-domain}
and~\ref{remark:matrix-Ap-domain}, given a pair of scalar weights
$(w,v)$ on $\Omega$, we will assume that they are in fact defined and
locally integrable on a larger domain $\Omega'$ such that
$\Omega \Subset \Omega'$ and that balls and cubes are interchangeable
in the definition of doubling and $A_p$ weights on $\Omega$.  Note
that as a consequence of this assumption, $v,\,w \in L^1(\Omega)$.

\begin{definition} \label{defn:padm}Given $1<p<\infty$, a domain
  $\Omega$ and a pair of weights $(w,v)$, we say that the pair is
  $p$-adimissible on $\Omega$ if:
\begin{enumerate}
\item $w\leq v$;
\item $w\in A_p$;
\item $v$ is doubling;
\item $(w,v)$ satisfies the balance condition:  there exists $q>p$
  such that for every ball $B\subset \Omega$ and $0<r<1$, 
\begin{equation} \label{balance}
r\Big(\frac{v(rB)}{v(B)}\Big)^{1/q}\lesssim \Big(\frac{w(rB)}{w(B)}\Big)^{1/p}.
\end{equation}
\end{enumerate}
\end{definition}

If $W$ is an invertible matrix weight, then by Proposition~\textup{\ref{prop:elliptic}} we have that $(1)$ holds.  Since
$p>1$, if we further assume $W\in \A_p$, then $(2)$ and $(3)$ hold by
Lemma~\textup{\ref{lemma:scalar-Ap}}.   Therefore, the critical condition is
the balance condition $(4)$.  We will consider this assumption more
carefully in Section~\textup{\ref{section:balance}}.

Given the assumption that $(w,v)$ are a $p$-admissible pair, then by
Theorem~\ref{thm:complete} we have $\W_W^{1,p}(\Omega)$ is complete and
we can take the solution space of~\eqref{eqn:new-PDE} to be $\Sp_W^{1,p}(\Omega)$.  More
precisely (again following~\cite{MR3011287}) we define a weak
solution of~\eqref{eqn:new-PDE} to be a function $u\in
\Sp_W^{1,p}(\Omega)$ such that for all $\vp \in C_c^\infty(\Omega)$, 
$$\int_\Omega |W^{1/p}\nabla u|^{p-2}\langle W^{1/p}\nabla
u,W^{1/p}\nabla \vp\rangle\,dx=0.$$ 
Note that if we only assume $u\in \W_W^{1,p}(\Omega)$, then the above
integral is well defined even if $\vp\in \W_W^{1,p}(\Omega)$.  Indeed,
if we apply the Cauchy-Schwarz inequality and then H\"older's
inequality, we get
\begin{multline*}
\int_\Omega \big||W^{1/p}\nabla u|^{p-2}\langle W^{1/p}\nabla u,W^{1/p}\nabla \vp\rangle\big|\,dx \\
\leq\int_\Omega |W^{1/p}\nabla u|^{p-1}|W^{1/p}\nabla \vp|\,dx
\leq \|\nabla u\|_{L^p_W(\Omega)}^{p-1}\|\nabla \vp\|_{L^p_W(\Omega)}.
\end{multline*}
As a consequence, we could define weak solutions $u$ to be functions
in $\W_W^{1,p}(\Omega)$.  However, to prove the results given below,
we need the stronger assumption that $u\in \Sp_W^{1,p}(\Omega)$.
However, if $W\in \A_p$ and $\Omega$ has Lipschitz boundary (e.g., if
$\Omega$ is a ball), then by Theorem~\ref{thm:bndry-smooth} we have
that $\Sp_W^{1,p}(\Omega)=\H_W^{1,p}(\Omega)=\W_W^{1,p}(\Omega)$, so
we can take our solution space to be either of these ``larger''
spaces.  We will use this fact in Section~\ref{section:balance} below.

The following results are from  \cite{MR3011287}. For brevity, in
the next result let
$\Sp_{W,0}^{1,p}(\Omega)$  denote the closure of $C_c^\infty(\Omega)$
in $\W_W^{1,p}(\Omega)$. 

\begin{theorem}[\cite{MR3011287} Theorem~3.11] \label{exist} 
Let  $1<p<\infty$ and let $W\in \Sp_n$ be an invertible matrix such that
  $(w,v)$ is a $p$-admissible pair.  Then given any
  $\psi\in \Sp_W^{1,p}(\Omega)$ there exists a weak solution
  $u\in \Sp_W^{1,p}(\Omega)$ of $\Lap_{W,p}u=0$ such that
  $u-\psi\in \Sp_{W,0}^{1,p}(\Omega)$.
\end{theorem}

\begin{remark}
In the original statement of this result  there is an
assumption that a global Sobolev inequality holds.  This was necessary
there because they were considering more general equations defined
with respect to H\"ormander vector fields.  Since~\eqref{eqn:new-PDE}
is defined with respect to the gradient, this assumption always
holds.  See the discussion in~\cite[Section~3]{MR3011287}.
\end{remark}

\begin{theorem}[\cite{MR3011287} Theorem 3.16] \label{bound} 
  If $B$ is a ball and $u\in \Sp_W^{1,p}(2B)$ is a weak solution of
  $\Lap_{W,p}u=0$, then $u$ is bounded on $B$.
\end{theorem}

\begin{theorem}[\cite{MR3011287} Theorem 3.17] \label{Harnack} 
  If $B$ is a ball and $u\in \Sp_W^{1,p}(2B)$ is a non-negative weak
  solution of $\Lap_{W,p}u=0$, then $u$ satisfies
  the following Harnack inequality:
\begin{equation}\label{eqn:Harnack} 
\sup_B u\leq \exp\big(C\mu(B)^{1/p}\big)\inf_B u
\end{equation}
where $\mu(B)=\frac{v(B)}{w(B)}$.
\end{theorem}

To state our next result, we introduce an auxiliary operator:  a
weighted maximal operator.   Given scalar weights $(w,v)$,
for $x\in \Omega$ define the maximal operator
\[ M_\Omega(w,v)(x) = \sup_{B} \frac{v(B)}{w(B)}, \]
where the supremum is taken over all balls $B\subset \Omega$ centered
at $x$.   Since $v\in L^1(\Omega)$, 
it follows from a Besicovitch covering lemma argument (cf. Journ\'e
\cite[Chapter~1]{MR706075}) that 
\[ w(\{ x \in \Omega : M_\Omega(w,v)(x)>\lambda\}) \leq
\frac{C_n}{\lambda}v(\Omega). \]
In particular, the set $\{ x\in \Omega :
M_\Omega(w,v)(x) = \infty \}$ has measure zero.

\begin{theorem} \label{thm:partial-reg}
Given $1<p<\infty$ and an invertible matrix weight $W\in \Sp_n$,
suppose that $(w,v)$ is a $p$-admissible pair.   If $u\in
\Sp_W^{1,p}(\Omega)$ is a weak solution of $\Lap_{W,p}u=0$,   then
$u$ is continuous on the set 
$$ F_\Omega(w,v) = \{x\in \Omega:M_\Omega(w,v)(x)<\infty\}.$$
In particular, $u$ is continuous almost everywhere in $\Omega$.
\end{theorem}

\begin{proof} 
Our proof follows closely the proofs in \cite[Theorems 4.4 and
  4.5]{MR3011287}, so here we only sketch the main ideas.   
  
Note that by the above discussion we have that 
$|F_\Omega(w,v)|=|\Omega|$.  
Fix $x\in F_\Omega(w,v)$
  and let $B=B(x,r)$ be a ball such that $2B=B(x,2r)\subset \Omega$.
  Since $u\in \Sp_W^{1,p}(\Omega)$ it is clearly in $\Sp_W^{1,p}(2B)$ and is a
  solution to $\Lap_{W,p}u=0$ on $2B$.  So by Theorem \ref{bound}
  $u$ is bounded on $B$.  Therefore, if we let $M$ and $m$ be upper
  and lower bounds for $u$ on $B$, we can apply the Harnack
  inequality \eqref{eqn:Harnack} to $M-u$ and $u-m$ to conclude
that
$$\osc_u(x,\small{\tfrac12}B)
\leq
\frac{\exp\big(C\mu(\frac12B)^{1/p}\big)-1}{\exp\big(C\mu(\frac12B)^{1/p}\big)+1}\osc_u(x,B),$$
where $\osc_u(x,B)=\sup_B u-\inf_Bu$ is the oscillation of $u$ on $B$.  

Since
$$\mu(\tfrac12B)=\frac{v(\frac12B)}{w(\frac12B)}\leq
M_\Omega(w,v)(x),$$
we have that
$$\frac{\exp\big(C\mu(\frac12B)^{1/p}\big)-1}{\exp\big(C\mu(\frac12B)^{1/p}\big)+1}
\leq
\frac{\exp\big(CM_\Omega(w,v)(x)^{1/p}\big)-1}{\exp\big(CM_\Omega(w,v)(x)^{1/p}\big)+1}
=\gamma(x).$$
Moreover,
$F_\Omega(w,v) =\{x\in \Omega:\gamma(x)<1\}$.
Because $\gamma(x)<1$ we may perform Moser iteration (see \cite[Lemma
8.23]{MR1814364}) to show there exists $0<s(x)<\infty$ such that
$$\osc_u(x,\al B)\leq c(x) \al^{s(x)}\osc_u(x,B), \qquad 0<\al<1.$$ 
It follows from this inequality that $u$ agrees a.e.~with a function that is continuous on $F_\Omega(w,v)$.
\end{proof}

\section{The balance condition}
\label{section:balance}

In this section we consider the partial regularity of solutions of the
degenerate $p$-Laplacian equation $\Lap_{W,p}u=0$ with the additional
assumption that $W\in \A_p$.  Since $W\in \A_p$ implies (by Lemma
~\ref{lemma:scalar-Ap}) $w,v\in A_p$,
we have that conditions (1), (2) and (3) of
Definition~\ref{defn:padm} hold.  However,  the balance condition
\eqref{balance} does not follow automatically from the matrix $\A_p$
condition, as the next example shows.

\begin{example}
We modify Example~\ref{example:inclusion}.   For ease of computation
we will consider the balance condition for cubes instead of balls, but
it is clear that they are interchangeable in this setting.  
 Let $\Omega = Q =
(0,1)\times (0,1)$.  Fix $1<p<2$, $q>p$ and $p/2<\alpha<1$.  We again define
\[ W(x,y) = \left[\begin{array}{cc} 
1  & 0 \\ 0 & x^{-\alpha}y^{-\alpha} 
\end{array}\right].  \]
Then as before $W\in \A_p$
and $w(x,y)=1$, $v(x,y)=x^{-\alpha}y^{-\alpha}$.  
Furthermore,
\[ r \left(\frac{v(rQ)}{v(Q)}\right)^{1/q} \approx
r^{\frac{2-2\alpha}{q}+1},
\qquad \left(\frac{w(rQ)}{w(Q)}\right)^{1/p} = r^{\frac{2}{p}}. \]
Therefore, the balance condition holds only if
\[ \frac{2-2\alpha}{q}+1 \geq \frac{2}{p}. \]
However, by  our choice of $p$, $q$ and $\alpha$,
\[ \frac{2-2\alpha}{q}+1 < \frac{2-p}{p} +1 = \frac{2}{p}. \]
\end{example}

\medskip

Given this example, we want to determine sufficient conditions on $W$,
or more precisely on $v$ and $w$, for the balance condition to hold.
Intuitively, the above example fails because our choice of $\alpha$ is
too close to $1$:  the function $x^{-\alpha}$ is in $A_1$, but it only
satisfies the reverse H\"older inequality for small values of $s>1$.  
Our main result, which is a generalization
of~\cite[Theorems~4.8,~4.9]{MR3011287}, shows that a sufficiently
large reverse H\"older exponent yields the balance condition.

\begin{theorem} \label{expbalance} 
Given $1<p<\infty$,   suppose $1<s,t<\infty$, $w\in A_t$ and $v\in
RH_s$ where 
\begin{equation}\label{exponents}
0<t-\frac{p}{n}= \frac{1}{s'}.
\end{equation}
Then $(w,v)$ satisfies the balance condition \eqref{balance}.
\end{theorem}

\begin{proof} 
Since $w\in A_t$ there exists $\ep>0$ such that $w\in A_{t-\ep}$
(see~\cite{duoandikoetxea01}).  In particular, by \eqref{exponents},
$$0<\frac{n}{p}(t-\ep)-1< \frac{n}{ps'}.$$
Define 
$$q=\frac{n/s'}{(t-\ep)n/p-1};$$
then we have that $q>p$. 
Fix  $0<r<1$ and a ball $B$.   By inequality~\eqref{eqn:Ainfty-Ap},
\begin{equation}\label{Apbd}
r^{\frac{n}{p}(t-\ep)}=\Big(\frac{|rB|}{|B|}\Big)^{\frac{t-\ep}{p}}
\leq C \Big(\frac{w(rB)}{w(B)}\Big)^{1/p}. 
\end{equation}
Moreover, by inequality~\eqref{eqn:Ainfty-RHs},
\begin{equation}\label{RHbd} 
r\Big(\frac{v(rB)}{v(B)}\Big)^{1/q}\leq
Cr\Big(\frac{|rB|}{|B|}\Big)^{1/(s'q)}
=Cr^{\frac{n}{s'q}+1}.
\end{equation}
If we combine \eqref{Apbd} and \eqref{RHbd} we immediately get the
balance inequality \eqref{balance}.
\end{proof}

\begin{remark} 
A close examination of the proof shows that it is enough to assume
that $1<s,t<\infty$ satisfy 
\begin{equation}\label{ineq:rt}
 0<t-\frac{p}{n}\leq \frac{1}{s'} 
\end{equation}
However, since the $A_t$ and $RH_s$ classes are nested (i.e., if
$u<t$, then $A_u\subseteq A_t$, and if $q>s$, $RH_q\subseteq RH_s$),
equality in \eqref{ineq:rt} is the interesting case.
\end{remark}

\medskip

Theorem~\ref{expbalance} seems to require a stronger condition on both
$v$ and $w$.  However, depending on the size of $p$ relative to the
dimension $n$, we can shift the stronger condition to one weight or
the other.  We first consider $p$ small:  in this case we require a
stronger condition on $v$.   

\begin{corollary} 
Suppose $1<p<n'$ and $W\in \A_p$.  If $v\in RH_{\frac{n'}{n'-p}}$,
then $(w,v)$ satisfy the balance condition \eqref{balance}. 
\end{corollary}

\begin{proof}  
Since $W \in \A_p$, we have $w\in A_p$.  Therefore, if we let
$s=\frac{n'}{n'-p}$, then $r'=n'/p$ and so
$$\frac{n}{p}p-1=n-1=\frac{n}{n'}=\frac{n}{ps'}.$$
Therefore, by Theorem~\ref{expbalance} the balance condition holds.
\end{proof}

When $p$ is large, we can shift the stronger hypothesis to  $w$.  The
following two corollaries  are immediate consequences of
Theorem~\ref{expbalance}. 

\begin{corollary} \label{RHassump} 
Suppose $p\geq n$, $W\in \A_p$ and $w\in A_t$, where 
$$\frac{p}{n}<t\leq\frac{p}{n}+\frac{1}{s'}$$ 
and $s>1$ is such that $v\in RH_s$.   Then the pair $(w,v)$ satisfies the balance
condition~\eqref{balance}. 
\end{corollary}

\begin{remark}
Since $v\in A_p$ we know that $v\in RH_s$ for some $s>1$, so there
exists some $t>1$ for which the hypotheses hold.   Indeed,
by~\eqref{eqn:sharp-RH}, we can give a sharp estimate for $t$:
\[ t\leq \frac{1}{2^{n+12}[v]_{A_\infty}}+\frac{p}{n}. \]
\end{remark}

To state the next result, let  $A_{q}^*=\bigcap_{p>q} A_p$.  Note that
this  class is strictly larger than $A_q$.

\begin{corollary} \label{cor:Ap*}
If $p\geq n$, $W\in \A_p$ and $w\in A_{p/n}^*$, then the pair $(w,v)$ satisfies the balance condition \eqref{balance}.
\end{corollary}

\medskip

As a consequence of Theorem~\ref{expbalance} and its corollaries, we
get the following partial regularity result.

\begin{theorem} \label{thm:Ap-degen}
If $1<p<\infty$, $W\in \A_p$, and $w,\,v$ satisfy the hypotheses of
any of the above results, and if $u$ is a weak solutions to
$\Lap_{W,p}u=0$, then $u$ is  continuous on the set
$$F_\Omega(w,v)= \{M_\Omega(w,v)(x)<\infty\}. $$
\end{theorem}

\begin{remark}
  Theorem~\ref{thm:Ap-degen} is the best possible: that is, there
  exists $W$ satisfying the hypotheses such that a solution to
  $\Lap_{W,p}u=0$ is discontinuous on the complement of
  $F_\Omega(w,v)$.  See Example~\ref{example:optimal} below.  
\end{remark}

\section{Mappings of finite distortion}
\label{section:finite-distortion}

In this section we apply our results on the partial regularity of
solutions of the
degenerate $p$-Laplacian to mappings of finite distortion.  
Hereafter, let $\Omega\subset \R^n$ be a domain that is not
necessarily bounded.  A vector function $\vf:\Omega\rightarrow \R^n$
is a mapping of finite distortion (MFD) if
\begin{enumerate}
\item $\vf\in \W_{\loc}^{1,1}(\Omega,\R^n)$; 
\item the Jacobian $J_{\vf}(x)=\det D\vf(x)>0$ a.e., and $ J_\vf \in
  L^1_{\loc}(\Omega)$; 
\item there exists $K(x)<\infty$ a.e.~such that
$|D\vf(x)|_\op^n\leq K(x)J_{\vf}(x)$.
\end{enumerate}
As we noted in the Introduction, the regularity of MFDs has been studied by a
number of authors.  A classical result due to Vodop'janov and
Gol'd{\v{s}}te{\u\i}n \cite{MR0414869} is that if
$\vf \in W^{1,n}(\Omega)$, then $\vf$ is continuous.   Generally
speaking, most results in this area show that if $\vf \not\in
W^{1,n}(\Omega)$,  continuity follows if
$\exp(K)$ satisfies some kind of integrability condition.    Our
results are quite different as we only prove partial regularity; they
are similar in spirit, though not in detail, to the work of
Manfredi~\cite{MR1294334}.

To state our results, we first give
some basic definitions and results on MFDs; for complete information,
including proofs, 
see~\cite{MR1207810,MR1859913}. 
The smallest function $K$ such that the (3) holds is called the outer
distortion of $\vf$ and is denoted $K_O$: i.e.,
 $$|D\vf(x)|_\op^n=K_O(x)J_{\vf}(x).$$
Since it is always the case that $|D\vf(x)|_\op\geq J_{\vf}(x)$, we
must have that $K_O(x)\geq 1$ a.e.
Similarly, we define the inner distortion, denoted $K_I$, to be the
smallest distortion function of the inverse differential matrix:
$$|D\vf^{-1}(x)|_\op^n=K_I(x)J_{\vf^{-1}}(x)=K_I(x)J_{\vf}(x)^{-1}.$$
%
The inner and outer distortion functions are related by the
inequalities
$$K_O\leq K_I^{n-1} \qquad \text{and} \qquad K_I\leq K_O^{n-1}.$$
Finally, if we define the maximal distortion $K_M$ by
\[ K_M(x) = \max\big( K_I(x), K_O(x)\big),  \]
then we have that 
$$K_I\leq K_M\leq K_I^{n-1} \qquad \text{and} \qquad K_O\leq K_M\leq
K_O^{n-1}. $$

\medskip

We now show that a mapping of finite distortion is a solution to a
degenerate $p$-Laplacian equation.  Define the distortion tensor of $\vf$ to be the  symmetric matrix
$$G(x)=J_{\vf}(x)^{-2/n} D\vf(x)^\T D\vf(x).$$
Hereafter, let $W=G^{-n/2}$.  Then we have that
\begin{equation} \label{eqn:KO-w}
|W^{-1}(x)|_\op=\frac{|D\vf(x)^\T
  D\vf(x)|_\op^{n/2}}{J_{\vf}(x)}=K_O(x)
\end{equation}
and
\begin{equation} \label{eqn:KI-v}
|W(x)|_\op=|D\vf(x)^{-1}
(D\vf(x)^{-1})^\T|_\op^{n/2}J_{\vf}(x)=K_I(x).
\end{equation}
In particular, by inequality~\eqref{eqn:elliptic},
$$K_O(x)^{-1}|\bxi|^n
\leq |W^{1/n}\bxi|^n\leq K_I(x){|\bxi|^n}.$$

Let $\vf=(f_1,\ldots,f_n)$.  Then by definition, $f_i \in
\W^{1,1}_\loc(\Omega)$.   Suppose  $(K_O^{-1},K_I)$ is an
$n$-admissible pair.  Then  given any ball
$B\subseteq\Omega$,   we have that $f_i \in \W^{1,n}_W(B)$.  
We first show that $\grad f_i \in L^n_W(B)$.   Given a matrix $A$, let
$[A]_i$ denote its $i$-th column.  Then, treating $\grad f_i$ as a
column vector, we have that 
$$ D\vf^{-1}(D\vf^{-1})^\T \nabla f_i=D\vf^{-1}\boldsymbol{e}_i=[D\vf^{-1}]_i $$
and $[D\vf(x)^{-1}]_i\cdot\nabla f_i=1$. Therefore,
$$\int_B |W^{1/n}\nabla f_i|^n\,dx=\int_B \langle G^{-1}\nabla f_i,\nabla f_i\rangle^{n/2}\,dx=\int_B J_{\vf}(x)\,dx<\infty.$$
To show that $f_i \in L^n(K_I,B)$, we use the fact that since
$(K_O^{-1},K_I)$ is an $n$-admissible pair, we have a two-weight
poincare inequality (see~\cite{MR805809}):
\begin{multline*}\frac{1}{K_I(B)}\int_B|f_i-(f_i)_B|^nK_I\,dx\lesssim \frac{r(B)^n}{K_O^{-1}(B)}\int_B|\nabla f_i|^n K_O^{-1}\,dx\\
\leq \frac{r(B)^n}{K_O^{-1}(B)}\int_B|W^{1/n}\nabla f_i|^n\,dx=\frac{r(B)^n}{K_O^{-1}(B)}\int_B J_{\vf}(x)\,dx.\end{multline*}
It follows that 
$$\|f_i\|_{L^{n}_W(B)}\lesssim \|f_i\|_{L^1(B)} +\Big(\int_B J_{\vf}\,dx\Big)^{1/n}.$$

Finally, we have that if $\vf\in \W_{\loc}^{1,n-1}(\Omega,\R^n)$, then
the component functions, $f_i$ are weak solutions
of
$$\Lap_{W,n}u=\Div(|W^{1/n}\nabla u|^{n-2}W^{2/n}\nabla u)=0.$$
See~\cite[Chapter~15]{MR1859913} for details.



From these observations we see that given an MFD $\vf$, we have that
the component functions $f_i$, $1\leq i \leq n$, satisfy a degenerate
$p$-Laplacian equation $\Lap_{W,n}u=0$, where the matrix $W$ satisfies
the natural ellipticity conditions with bounds given by the distortion
functions.  In other words, these functions fall within the framework
of our results in the previous two sections.  This leads to the
following partial regularity result for mappings of finite distortion.

\begin{theorem} \label{thm:MFD-pr} 
  Given an MFD, $\vf \in \W^{1,n-1}_{\loc}(\Omega,\R^n)\cap \Sp_W^{1,n}(\Omega)$, suppose
  $(K_O^{-1}, K_I)$ is an
  $n$-admissible pair.    Then $\vf$ is continuous almost everywhere
  on $\Omega$.  More precisely, given any ball $B\subset \Omega$, then
  $\vf$ is continuous on the set
$$\{x\in B: M_B(K_O^{-1},K_I)(x)<\infty\}.$$
\end{theorem}

\begin{proof} 
  Fix a ball $B\subset \Omega$.  The component functions of $\vf$
  belong to $\Sp_W^{1,p}(B)$ and are weak solutions of
  $\Lap_{W,p}u=0$.  Therefore, by Theorem \ref{thm:partial-reg} each
  of the component functions is continuous on the set
\begin{equation} \label{eqn:cont-set}
F_B(K_O,K_I)=\{x\in B: M_B(K_O^{-1},K_I)(x)<\infty\}.
\end{equation}
Since $|F_B(K_O,K_I)|=|B|$ and $\Omega$ is the countable
union of balls, we have that $\vf$ is continuous almost everywhere on
$\Omega$.
\end{proof}

\medskip

Following our approach in Section~\ref{section:balance}, we now
consider the hypothesis that  $(K_O^{-1}, K_I)$ is an
  $n$-admissible pair given the additional assumption that $W=G^{-n/2}
  \in \A_n$.  (Equivalently, we may assume $W^{-n'/n}=G^{n'/2}\in
  A_{n'}$.  This is particularly useful when $n=2$.)  In this case, by Lemma~\ref{lemma:scalar-Ap} and the
identities~\eqref{eqn:KO-w} and~\eqref{eqn:KI-v}, we have that
conditions (1) and (2) of Definition~\ref{defn:padm} hold, so the main
problem is determining additional assumptions so that the balance
condition~\eqref{balance} holds.  Our first result is just a
restatement of Theorem~\ref{expbalance} in this setting.

\begin{corollary} \label{cor:MFD-wv-cond}
Suppose $\vf$ is an MFD and $W=G^{-n/2}\in \A_n$.  Suppose further
that $K_O^{-1} \in A_t$ and $K_I \in RH_s$, where 
\[ 0 < t-1 = \frac{1}{s'}. \]
Then $\vf$ is continuous a.e.~on $\Omega$ and the set of
continuity is given by~\eqref{eqn:cont-set}.
\end{corollary}

Since in our setting $p=n$, we can apply Corollaries~\ref{RHassump}
and~\ref{cor:Ap*}.  For brevity we will only consider the latter and
leave the restatement of the former to the interested reader.  

\begin{corollary} \label{cor:KO-A1*}
  Suppose $\vf \in \W_{\loc}^{1,n-1}(\Omega,\R^n)$ is an MFD,
  $W=G^{-n/2}\in \A_n$, and
$$1/K_O\in A_1^*=\bigcap_{p>1} A_p.$$  
Then $\vf$ is continuous a.e.~on $\Omega$ and the set of
continuity is given by~\eqref{eqn:cont-set}.
\end{corollary}

As a consequence of Corollary~\ref{cor:KO-A1*} we give two results
which implicitly require the outer distortion to be exponentially
integrable.  For the first result, note that a weight $w$ is such that
$w,\,w^{-1} \in A_1^*$ if and only if $\log(w)$ is in the closure of
$L^\infty$ in $BMO$; in particular, the latter inclusion holds if
$\log(w) \in VMO$.
(See~\cite[p.~474]{garcia-cuerva-rubiodefrancia85}.)

\begin{corollary} 
Suppose $\vf \in W_{\loc}^{1,n-1}(\Omega,\R^n)$, $W=G^{-n/2}\in \A_n$,
and $\log(K_O)$ is in the closure of $L^\infty$ in $BMO$.  
Then $\vf$ is continuous a.e.~on $\Omega$ and the set of
continuity is given by~\eqref{eqn:cont-set}.
\end{corollary}

For the second result, we use the fact that if $b$ is a function such
that $b,1/b\in BMO$, then $b\in A_1^*$.  (See~\cite{MR1239426}.)
Since $K_O\geq 1$,  we always have
that $K_O^{-1}\in L^\infty(\Omega)\subset BMO$.

\begin{corollary} 
Suppose $\vf \in W_{\loc}^{1,n-1}(\Omega,\R^n)$, $W=G^{-n/2}\in \A_n$, and
$K_O\in BMO$.  Then $\vf$ is continuous a.e.~on $\Omega$ and the set of
continuity is given by~\eqref{eqn:cont-set}.
\end{corollary}



\medskip

We now want to give some partial regularity theorems that are related
to the results in~\cite{MR3011287}.  The major improvement here is
that by assuming that $W$ is in matrix $\A_n$ we no longer have to
assume that a weak solution is in the closure of the smooth
functions.   To state our results we first note that in
Theorem~\ref{thm:partial-reg}, while we implicitly assumed that
$v=|W|_\op$ and $w=|W^{-1}|_\op^{-1}$, we never used this in the
proof.  All we used was the fact that $(w,v)$ is a $p$-admissible
pair, and the ellipticity condition~\eqref{ellipcond} holds.
Further, note that using the relationships relating them, we can give
ellipticity conditions for $W=G^{-n/2}$ in terms of the distortion
functions $K_M$, $K_O$ and $K_I$:
\begin{align}
&K_M(x)^{-1}|\boldsymbol{\xi}|^n\leq
  |W^{1/n}\boldsymbol{\xi}|^n\leq {K_M(x)}|\boldsymbol{\xi}|^n 
\label{maxdist}\\
&K_O(x)^{-1}|\boldsymbol{\xi}|^n\leq
  |W^{1/n}\boldsymbol{\xi}|^n\leq K_O(x)^{n-1}|\boldsymbol{\xi}|^n 
\label{outerdist}\\
&K_I(x)^{1-n}|\boldsymbol{\xi}|^n\leq
  |W^{1/n}\boldsymbol{\xi}|^n\leq {K_I(x)}|\boldsymbol{\xi}|^n. 
 \label{innerdist}
\end{align}

To state our results we will need a local version of the
Hardy-Littlewood maximal operator.   Given a ball $B\subset \Omega$
and a locally integrable function $f$ define
$$M_Bf(x)=\sup_{B'}\avgint_{B'}|f(y)|\,dy \cdot
\chi_{B'}(x),$$
where the supremum is over all  balls $B'\subset B$.

\begin{theorem} \label{thm:KM-reg}
  Given an MFD $\vf \in W_{\loc}^{1,n-1}(\Omega,\R^n)$, suppose
  $W=G^{-n/2}\in \A_n$, and $K_M\in A_2\cap RH_2$.  Then $\vf$ is
  continuous almost everywhere on $\Omega$.  More precisely, given any
  ball $B$, $\vf$ is continuous on the set
$$\{x\in B: M_B(K_M)(x)<\infty\}.$$
\end{theorem}

\begin{proof} 
Since we have the ellipticity condition \eqref{maxdist}, to apply Theorem~\ref{thm:partial-reg} we need to
show that $(K_M^{-1},K_M)$ is an $n$-admissible pair.  Since $K_M\in
A_2$, $K_M^{-1}\in A_2$, and so conditions (1) and (2) in
Definition~\ref{defn:padm} hold.   Since $K_M\in A_2\cap RH_2$, we have that
$K_M^2\in A_3$ (see~\cite[Theorem~2.2]{MR1308005}), which by the duality of $A_p$
weights implies that $K_M^{-1}\in A_{3/2}$.  We can therefore apply
Theorem~\ref{expbalance} with $t=3/2$ and $s=2$ to conclude that
$(K_M^{-1},K_M)$ satisfy the balance condition~\eqref{balance}.
Moreover, by H\"older's inequality, we have that for any ball
$B\subset \Omega$,
$$M_B(K_M^{-1},K_M)(x)\leq M_B(K_M)(x)^2, $$
so 
$$\{x\in B:M_B(K_M)(x)<\infty\}\subset \{x\in B:
M_B(K_M^{-1},K_M)(x)<\infty\}.$$ 
The desired conclusion now follows from Theorem~\ref{thm:partial-reg}.
\end{proof}

\begin{theorem} 
  Given an MFD $\vf \in W_{\loc}^{1,n-1}(\Omega,\R^n)$, suppose
  $W=G^{-n/2}\in \A_n$, and $K_I\in A_{n'}\cap RH_n$.  Then $\vf$ is
  continuous almost everywhere on $\Omega$:  given any ball
  $B\subset \Omega$, $\vf$ is continuous on the set
$$\{x\in B: M_B(K_I)(x)<\infty\}.$$
\end{theorem}

\begin{proof} 
  We proceed as in the proof of Theorem~\ref{thm:KM-reg}:  given the
  ellipticity condition \eqref{innerdist}, it will suffice to show
  that 
  $(K_I^{1-n},K_I)$ is an $n$ admissible pair.  Since
  $K_I\in A_{n'}\cap RH_n$, we have that  $K_I^n \in A_{n'+1}$, 
  which in turn implies that 
$$K_I^{1-n}=(K_I^{n})^{-\frac{1}{n'}}\in A_{1+\frac{1}{n'}}.$$
Hence conditions (1) and (2) hold.  Moreover, if we take
$t=1+\frac{1}{n'}$ and $s=n$ in Theorem~\ref{expbalance}, we see that
the weights satisfy the balance condition.  Finally, 
$$M_B(K^{1-n}_I,K_I)(x)\leq M_B(K_I)(x)^n.$$
\end{proof}

\begin{theorem} 
  Given an MFD $\vf$, suppose $W=G^{-n/2}\in \A_n$, and
  $K_O^{n-1}\in A_{n}\cap RH_{n'}$.  Then $\vf$ is continuous almost
  everywhere on $\Omega$:   given any  ball $B\subset\Omega$, $\vf$ is
  continuous on the set
$$\{x\in B: M_{B,n-1}(K_O)(x)<\infty\},$$
where $M_{B,n-1}(K_I)=M_B(K_O^{n-1})^{1/(n-1)}$.
\end{theorem}

\begin{proof} 
  First, by our assumption $K_O\in L^{n-1}_{\loc}(\Omega)$, we
  do not need to assume {\it a priori} that
  $\vf \in \W^{1,n-1}_\loc(\Omega,\R^n)$.  Indeed, since
  $J_\vf(x)$ is locally integrable, if $B\subset\Omega$, then
\begin{align*}
\int_B |D \vf|_\op^{n-1}\,dx
&=\int_B |D\vf \, W^{1/n}W^{-1/n}|_\op^{n-1}\,dx \\
&\leq \int_B |D\vf \, W^{1/n}|^{n-1}_\op |W^{-1/n}|_\op^{n-1}\,dx \\
&=\int_B |D\vf \, W^{1/n}|^{n-1}_\op K_O^{1/n'}\,dx\\
&\leq\Big(\int_B |D\vf \, W^{1/n}|^n_\op\,dx\Big)^{1/n'}\Big(\int_B K_O^{n-1}\,dx\Big)^{1/n}\\
&\lesssim\Big(\int_B J_\vf(x)\,dx\Big)^{1/n'}\Big(\int_B
  K_O^{n-1}\,dx\Big)^{1/n}<\infty. 
\end{align*}
For the last inequality we use the Frobenius norm,
$|A|_F=\sqrt{\tr(A^\T A)} \approx |A|_\op$, to get
$$|D\vf \, W^{1/n}|^n_\op\leq |D\vf G^{-1/2}|_F^n= \tr[(D\vf G^{-1/2})^\T (D\vf G^{-1/2})]^{n/2}=n^{n/2}J_\vf(x).$$

We can now argue again as in the proof of Theorem~\ref{thm:KM-reg}
using the ellipticity condition~\eqref{outerdist}.
Since $K_O^{n-1}\in A_n\cap RH_{n'}$ implies that $K^n_O\in A_{n+1}$,
by duality we have that $K_O^{-1}\in A_{1+1/n}$.  This gives
conditions (1) and (2).  If we take $t=1+\frac1n$ and $s=n'$ in
Theorem~\ref{expbalance}, then $(K_O^{-1},K_O^{n-1})$ satisfies the balance
condition.  Finally, again by H\"older's inequality,
$$M_B(K_O^{-1},K_O^{n-1})(x)\leq M_B(K_O^{n-1})(x)^{1/(n-1)}.  $$
\end{proof}

\medskip

We conclude this section with an example to show that our results are
sharp.  Our example is adapted from an example due to
Ball~\cite[Example~6.1]{MR0475169}.  As in all problems involving the
matrix $\A_p$ weights, the difficulty is in showing that the matrix is
in this class.  However, when $n=2$, we can use a result due to Lauzon
and Treil~\cite{MR2354705} to simplify the computations.

\begin{example} \label{example:optimal}
Fix $n=2$ and let $\Omega = B(0,1)$.  Define 
\[ \vf(x) =  (|x|^{-1}+|x|^{-1/2})x,  \quad x \neq 0, \]
and let $\vf(0)=0$.  Then $\vf$ maps $B(0,1)$ to the annulus
$B(0,2)\setminus B(0,1)$; clearly no choice of value for $\vf(0)$ will
make $\vf$ continuous there.  

We will show that $\vf$ satisfies the hypotheses and conclusions of
Corollary~\ref{cor:MFD-wv-cond}.  As shown in~\cite{MR0475169},
$f\in W^{1,1}_\loc(\Omega)$, and if we let $x=(x_1,x_2)$, $r=|x|$ and
$R=1+|x|^{1/2}$, then
\[ D\vf (x) = 
\begin{pmatrix}
\frac{R}{r}+ \frac{(rR'-R)x_1^2}{r^3} & \frac{(rR'-R)x_1x_2}{r^3} \\
\frac{(rR'-R)x_1x_2}{r^3}  & \frac{R}{r}+ \frac{(rR'-R)x_2^2}{r^3}\\
\end{pmatrix}
\]
and
\[ J_\vf (x) = \det D\vf(x) = \frac{RR'}{r}. \]
The eigenvalues of this matrix are (via Mathematica)
\begin{gather*}
 \mu_1 = \frac{r^2 R+r R' x_1^2+r R' x_2^2-R x_1^2-R x_2^2}{r^3} 
= R'=\frac{1}{2|x|^{1/2}}, \\
\mu_2 = \frac{R}{r} = \frac{1}{|x|}+\frac{1}{|x|^{1/2}}. 
\end{gather*}
Therefore (since $n=2$) we have that
\[ K_O(x)=K_I(x) = \frac{\mu_2}{\mu_1} = 2+\frac{2}{|x|^{1/2}}, \]
and so 
\[ K_I(x)  \approx |x|^{-1/2},  \qquad
 K_O(x)^{-1} \approx |x|^{1/2}.\]
Thus $K_O^{-1} \in A_t$ for $t>5/4$, and $K_I\in RH_s$ for $s<4$.
Therefore, we can take $t=3/2$ and $s=2$ and we satisfy the condition
$t-1=1/s'$.  

Therefore, it remains to show that $W=G^{-1} \in \A_2$.
In~\cite[Theorem~3.1]{MR2354705} they showed that this is the case if 
$\langle W(x)\bv,\bv\rangle$ and $\langle W^{-1}(x)\bv,\bv\rangle$ are
uniformly in scalar $A_2$ for all unit vectors $\bv\in \R^2$.   We
first consider $W^{-1}(x)= G(x)= J_\vf(x)^{-1} D\vf(x) D\vf(x)$.  This
matrix has eigenvalues
\[ \lambda_1 = J_\vf^{-1} \mu_1^2 =  \frac{|x|^{1/2}}{2(1+|x|^{1/2})}
\approx |x|^{1/2}, 
\qquad
\lambda_2 = J_\vf^{-1} \mu_2^2 = \frac {2(1+|x|^{1/2})}{|x|^{1/2}} \approx
|x|^{-1/2}.  \]
Therefore, $\lambda_1,\,\lambda_2 \in A_2$.  Given any unit vector
$\bv$, we can write it as $\alpha_1 \bxi_1+\alpha_2\bxi_2$, where
$\bxi_1,\,\bxi_2$ are an orthonormal basis of eigenvectors of
$W^{-1}$.  Hence,
\begin{equation} \label{eqn:weight}
\langle W^{-1}(x)\bv,\bv\rangle =
\alpha_1^2\lambda_1(x)+\alpha_2^2\lambda_2(x).  
\end{equation}
The linear combination of  two scalar $A_2$ weights is again an
$A_2$ weight, and its $A_2$
characteristic is dominated by
$\alpha_1^2[\lambda_1]_{A_2}+\alpha_2^2[\lambda_2]_{A_2}$.
(See~\cite[p.~292]{MR2463316}.)  Since $|\bv|=1$,  we get
that~\eqref{eqn:weight} is uniformly in $A_2$.  The argument for $W$
is exactly the same, using the fact that its eigenvalues are
$\lambda_1^{-1}$ and $\lambda_2^{-2}$, and these are again in $A_2$.
This completes our proof.
\end{example}

\bibliographystyle{plain}
\bibliography{H=W}

\end{document}